\documentclass[12pt]{amsart}

\usepackage{amsfonts}
\usepackage{amsmath}
\usepackage{amssymb}
\usepackage{latexsym}
\usepackage{amscd}
\usepackage{amsthm}
\usepackage[dvips]{graphicx}
\usepackage{subfigure}
\usepackage{layout}
\usepackage{array}
\usepackage{tabularx}
\usepackage[all]{xy}
\usepackage{color}

\newcommand{\Z}{\mathbb{Z}}
\newcommand{\B}{\mathcal{B}}
\newcommand{\gb}{\Gamma_2\mathcal{B}}
\newcommand{\la}{\langle}
\newcommand{\ra}{\rangle}
\newcommand{\g}{\Gamma_2}
\newcommand{\hg}{\widehat{\Gamma}}

\newcommand{\aaa}{(E_{12})_{v_1}}
\newcommand{\aab}{(E_{13})_{v_1}}
\newcommand{\aac}{(E_{23})_{v_1}}
\newcommand{\aad}{(E_{32})_{v_1}}
\newcommand{\aae}{(F_{2})_{v_1}}
\newcommand{\aaf}{(F_{3})_{v_1}}

\newcommand{\bba}{(E_{21})_{v_2}}
\newcommand{\bbb}{(E_{23})_{v_2}}
\newcommand{\bbc}{(E_{13})_{v_2}}
\newcommand{\bbd}{(E_{31})_{v_2}}
\newcommand{\bbe}{(F_{1})_{v_2}}
\newcommand{\bbf}{(F_{3})_{v_2}}

\newcommand{\cca}{(E_{31})_{v_3}}
\newcommand{\ccb}{(E_{32})_{v_3}}
\newcommand{\ccc}{(E_{12})_{v_3}}
\newcommand{\ccd}{(E_{21})_{v_3}}
\newcommand{\cce}{(F_{1})_{v_3}}
\newcommand{\ccf}{(F_{2})_{v_3}}

\newcommand{\dda}{(E_{21}F_2E_{12}F_1)_{v_4}}
\newcommand{\ddb}{(E_{13}E_{23})_{v_4}}
\newcommand{\ddc}{(E_{23})_{v_4}}
\newcommand{\ddd}{(E_{31}^{-1}E_{32})_{v_4}}
\newcommand{\dde}{(E_{21}F_{2})_{v_4}}
\newcommand{\ddf}{(F_{3})_{v_4}}

\newcommand{\eea}{(E_{31}F_3E_{13}F_1)_{v_5}}
\newcommand{\eeb}{(E_{12}E_{32})_{v_5}}
\newcommand{\eec}{(E_{32})_{v_5}}
\newcommand{\eed}{(E_{21}^{-1}E_{23})_{v_5}}
\newcommand{\eee}{(E_{31}F_{3})_{v_5}}
\newcommand{\eef}{(F_{2})_{v_5}}

\newcommand{\ffa}{(E_{32}F_3E_{23}F_2)_{v_6}}
\newcommand{\ffb}{(E_{21}E_{31})_{v_6}}
\newcommand{\ffc}{(E_{31})_{v_6}}
\newcommand{\ffd}{(E_{12}^{-1}E_{13})_{v_6}}
\newcommand{\ffe}{(E_{32}F_{3})_{v_6}}
\newcommand{\fff}{(F_{1})_{v_6}}

\newcommand{\gga}{(E_{21}F_2E_{12}F_1E_{31}^{-1}E_{32})_{v_7}}
\newcommand{\ggb}{(E_{31}F_3E_{13}F_1E_{21}^{-1}E_{23})_{v_7}}
\newcommand{\ggc}{(E_{21}^{-1}E_{23})_{v_7}}
\newcommand{\ggd}{(E_{31}^{-1}E_{32})_{v_7}}
\newcommand{\gge}{(E_{21}F_{2})_{v_7}}
\newcommand{\ggf}{(E_{31}F_{3})_{v_7}}

\newcommand{\xaaa}{E_{12}}
\newcommand{\xaab}{E_{13}}
\newcommand{\xaac}{E_{23}}
\newcommand{\xaad}{E_{32}}
\newcommand{\xaae}{F_{2}}
\newcommand{\xaaf}{F_{3}}

\newcommand{\xbba}{E_{21}}
\newcommand{\xbbd}{E_{31}}
\newcommand{\xbbe}{F_{1}}

\newcommand{\xij}{E_{12}}
\newcommand{\xik}{E_{13}}
\newcommand{\xjk}{E_{23}}
\newcommand{\xkj}{E_{32}}
\newcommand{\xxj}{F_{2}}
\newcommand{\xxk}{F_{3}}
\newcommand{\xji}{E_{21}}
\newcommand{\xki}{E_{31}}

\newcommand{\eij}{E_{1j}}
\newcommand{\eji}{E_{j1}}
\newcommand{\eik}{E_{1k}}
\newcommand{\eki}{E_{k1}}
\newcommand{\ejk}{E_{jk}}
\newcommand{\ekj}{E_{kj}}
\newcommand{\ffi}{F_{1}}
\newcommand{\ffj}{F_{j}}
\newcommand{\ffk}{F_{k}}

\DeclareMathOperator{\st}{\mathrm{star}}
\DeclareMathOperator{\lk}{\mathrm{link}}
\DeclareMathOperator{\R}{\mathrm{Rank}}

\newtheorem{thm}{Theorem}[section]
\newtheorem{prop}[thm]{Proposition}
\newtheorem{lem}[thm]{Lemma}
\newtheorem{cor}[thm]{Corollary}
\theoremstyle{example}
\newtheorem{rem}[thm]{Remark}

\address{Department of Mathematics \endgraf 
Faculty of Science and Technology \endgraf 
Tokyo University of Science \endgraf 
Noda, Chiba, 278-8510, Japan}
\email{kobayashi\_ryoma@ma.noda.tus.ac.jp}

\begin{document}

\title[A finite presentation of $\g(n)$]{A finite presentation of the level $2$ principal congruence subgroup of $GL(n;\Z)$}
\author[R. Kobayashi]{Ryoma Kobayashi}
\maketitle

\begin{abstract}
It is known that the level $2$ principal congruence subgroup of
 $GL(n;\Z)$ has a finite generating set (see \cite{mp}). 
In this paper, we give a finite presentation of the level $2$ principal
 congruence subgroup of $GL(n;\Z)$.
\end{abstract}

\section{Introduction}\label{intro}

For $n\geq1$, let $\g(n)=\ker(GL(n;\Z)\to{}GL(n;\Z_2))$.
We call $\g(n)$ the {\it level $2$ principal congruence subgroup} of
$GL(n;\Z)$.
Note that for $A\in\g(n)$ the diagonal entries of $A$ are odd and the
others are even.

For $1\leq{}i,j\leq{}n$ with $i\neq{}j$, let $E_{ij}$ denote the matrix
whose $(i,j)$ entry is $2$, diagonal entries are $1$ and others are $0$,
and let $F_i$ denote the matrix whose $(i,i)$ entry is $-1$, other
diagonal entries are $1$ and others are $0$.
It is known that $\g(n)$ is generated by $E_{ij}$ and $F_i$ for
$1\leq{}i,j\leq{}n$ with $i\neq{}j$ (see \cite{mp}).

In this paper, we give a finite presentation of $\g(n)$.

\begin{thm}\label{thm}
For $n\geq 1$,
$\g(n)$ has a finite presentation with generators $E_{ij}$ and $F_i$,
 for $1\leq{}i,j\leq{}n$ with $i\neq{}j$, and with the following
 relators
\begin{enumerate}
 \item $F_i^2$,
 \item $(E_{ij}F_i)^2$,
       $(E_{ij}F_j)^2$,
       $(F_iF_j)^2$
       (when $n\geq2$),
 \item \begin{enumerate}
	\item $[E_{ij},E_{ik}]$,
	      $[E_{ij},E_{kj}]$,
	      $[E_{ij},F_k]$,
	      $[E_{ij},E_{ki}]E_{kj}^2$ 
	      (when $n\geq3$),
	\item $[E_{ji}F_jE_{ij}F_iE_{ki}^{-1}E_{kj},E_{ki}F_kE_{ik}F_iE_{ji}^{-1}E_{jk}]$
	      for $i<j<k$ (when $n\geq3$),
       \end{enumerate}
 \item $[E_{ij},E_{kl}]$ (when $n\geq4$),
\end{enumerate}
where $[X,Y]=X^{-1}Y^{-1}XY$ and $1\leq{i,j,k,l}\leq{n}$ are  mutually
 different.
\end{thm}

We note that a finite presentation of $\g(n)$ has been obtained also by
Fullarton \cite{f} and Margalit-Putman.

It is clear that the above theorem is valid in the case $n=1$.
A proof of the theorem is by induction on $n$.
In Section~\ref{n=2}, we will prove the case $n=2$ of Theorem~\ref{thm},
using the Reidemeister-Schreier method.
In Section~\ref{n=3}, we will prove the case $n=3$ of Theorem~\ref{thm},
considering a simply connected simplicial complex on which $\g(n)$ acts.
In Section~\ref{gb}, we will introduce another simply connected
simplicial complex on which $\g(n)$ acts for $n\geq4$.
Finally, in Section~\ref{proof}, we will obtain the presentation of
Theorem~\ref{thm}, by this action and induction on $n$.

We now explain about an application of Theorem~\ref{thm}.
For $g\geq1$, let $N_g$ denote a non-orientable closed surface of genus
$g$, that is, $N_g$ is a connected sum of $g$ real projective planes.
Let $\cdot:H_1(N_g;R)\times{}H_1(N_g;R)\to\Z_2$ denote the mod $2$
intersection form, and let ${\rm{}Aut}(H_1(N_g;R),\cdot)$ denote the
group of automorphisms over $H_1(N_g;R)$ preserving the mod $2$
intersection form $\cdot$, where $R=\Z$ or $\Z_2$.
Consider the natural epimorphism 
$$\Phi_g:{\rm{}Aut}(H_1(N_g;\Z),\cdot)\to{\rm{}Aut}(H_1(N_g;\Z_2),\cdot).$$
McCarthy and Pinkall \cite{mp} showed that $\g(g-1)$ is isomorphic
to $\ker\Phi_g$.

We denote by $\mathcal{M}(N_g)$ the group of isotopy classes of
diffeomorphisms over $N_g$. 
The group $\mathcal{M}(N_g)$ is called the {\it mapping class group} of
$N_g$. 
In \cite{mp} and \cite{gp}, it is shown that the natural homomorphism 
$\mathcal{M}(N_g)\to{\rm{}Aut}(H_1(N_g;R),\cdot)$ is surjective, where
$R=\Z$ or $\Z_2$.
Let $\mathcal{I}(N_g)$ denote the kernel of 
$\mathcal{M}(N_g)\to{\rm{}Aut}(H_1(N_g;\Z),\cdot)$.
We say $\mathcal{I}(N_g)$ the {\it Torelli group} of $N_g$.
In \cite{hk}, Hirose and the author obtained a generating set of
$\mathcal{I}(N_g)$ for $g\geq4$, using Theorem~\ref{thm}.

\section{Preliminaries}\label{pre}

In this section, we explain about some facts for presentations of
groups.

\subsection{Basics on presentations of groups}\

Let $G_1,G_2$ and $G_3$ be groups with a short exact sequence
$$1\to{}G_1\stackrel{\phi}{\to}G_2\stackrel{\pi}{\to}G_3\to1.$$
If $G_1$ and $G_3$ are presented then we can obtain a presentation of
$G_2$.
In particular, if $G_1$ and $G_3$ are finitely presented then $G_2$ can
be finitely presented.

More precisely, a presentation of $G_2$ is obtained as follows.
Let $G_1=\la{}X_1\mid{}R_1\ra$ and $G_3=\la{}X_3\mid{}R_3\ra$.
For each $x\in{}X_3$, we choose $\widetilde{x}\in\pi^{-1}(x)$.
We put
$X_2=\{\phi(x_1),\widetilde{x_3}\mid{}x_1\in{}X_1,x_3\in{}X_3\}$.
For
$r=a_1^{\varepsilon_1}a_2^{\varepsilon_2}\cdots{}a_k^{\varepsilon_k}\in{}R_3$,
let
$\widetilde{r}=\widetilde{a_1}^{\varepsilon_1}\widetilde{a_2}^{\varepsilon_2}\cdots\widetilde{a_k}^{\varepsilon_k}$.
For $g\in\ker\pi$, let $\overline{g}$ be a word over $\phi(X_1)$ with
$g=\overline{g}$.
Let
$A=\{\phi(r_1)\mid{}r_1\in{}R_1\}$,
$B=\{\widetilde{r_3}\overline{\widetilde{r_3}}^{-1}\mid{}r_3\in{}R_3\}$
and
$C=\{\widetilde{x_3}\phi(x_1)\widetilde{x_3}^{-1}\overline{\widetilde{x_3}\phi(x_1)\widetilde{x_3}^{-1}}^{-1}\mid{}x_1\in{}X_1,x_3\in{}X_3\}$.
We put $R_2=A\cup{}B\cup{}C$.
Then we have $G_2=\la{}X_2\mid{}R_2\ra$.

In addition, if there is a homomorphism $\rho:G_3\to{}G_2$ such that
$\pi\circ\rho=id_{G_3}$, choose $\widetilde{x}=\rho(x)\in\pi(x)^{-1}$
for $x\in{}X_1$.
Then, we have the relation $\widetilde{r}=1$ in $G_2$ for $r\in{}R_3$.

If $G_2$ is presented then we can examine a presentation of $G_1$, by
the Reidemeister-Schreier method.
In particular, if $G_3$ is a finite group, that is, the index of
${\rm{}Im}\phi$ is finite, and $G_2$ can be finitely presented, then
$G_1$ can be finitely presented.

For further information see \cite{j}.

\subsection{Presentations of groups acting on a simplicial complex}\label{brown}\

Let $X$ be a simplicial complex, and let $G$ be a group acting on $X$ by
isomorphisms as a simplicial map.
We suppose that the action of $G$ on $X$ is {\it without rotation}, that
is, for a simplex $\Delta\in{}X$ and $g\in{}G$, if $g(\Delta)=\Delta$
then $g(v)=v$ for all vertices $v\in\Delta$.
For a simplex $\Delta\in{}X$, let $G_{\Delta}$ be the stabilizer of
$\Delta$.
For $k\geq0$, the {\it $k$-skeleton} $X^{(k)}$ is the subcomplex of $X$
consisting of all simplices of dimension at most $k$.

Consider a homomorphism $\Phi:\underset{v\in X^{(0)}}{\ast}G_v\to{}G$.
For $g\in{}G$, if $g$ stabilizes a vertex $w\in{}X^{(0)}$, we denote $g$
by $g_w$ as an element in $G_w<\underset{v\in{}X^{(0)}}{\ast}G_v$.
For a $1$-simplex $\{v,w\}\in{}X$ and $g\in{}G_{v}\cap{}G_{w}$, we have
$g_vg_w^{-1}\in\ker\Phi$ and call $g_vg_w^{-1}$ the {\it edge relator}.

At first, for any $1$-simplex $\{v,w\}$, choose an orientation such that
orientations are preserved by the action of $G$.
Namely, orientations of $\{v,w\}$ and $g\{v,w\}$ are compatible for
all $g\in{}G$.
We denote the oriented $1$-simplex $\{v,w\}$ by $(v,w)$.
Similarly, choose orders of $2$-simplices, and denote the ordered
$2$-simplex $\{v_1,v_2,v_3\}$ by $(v_1,v_2,v_3)$.
For an oriented $1$-simplex $e=(v,w)$, let $o(e)=v$ and $t(e)=w$.
For an oriented $2$-simplex $\tau=(v_1,v_2,v_3)$, we call $v_1$ the base
point of $\tau$.

Next, choose an oriented tree $T$ of $X$ such that a set of vertices of
$T$ is a set of representative elements for vertices of the orbit space
$G\backslash{}X$.
Let $V$ denote the set of vertices of $T$.
In addition, choose a set $E$ of representative elements for oriented
$1$-simplices of $G\backslash{}X$ such that $o(e)\in{}V$ for $e\in{}E$
and $1$-simplices of $T$ is in $E$, and a set $F$ of representative
elements for ordered $2$-simplices of $G\backslash{}X$ such that the
base point of $\tau$ is in $V$ for $\tau\in{}F$.
For $e\in{}E$, let $w(e)$ denote the element in $V$ which is equivalent
to $t(e)$ by the action of $G$, and choose $g_e\in{}G$ such that
$g_e(w(e))=t(e)$ and $g_e=1$ if $e\in{}T$.

For a $1$-simplex $\{v,w\}$ with $v\in{}V$, note that
$\{v,w\}=\{o(e),hg_ew(e)\}$ or $\{w(e),hg_e^{-1}o(e)\}$ for some
$e\in{}E$ and $h\in{}G_v$.
Then we define respectively $g_{\{v,w\}}=hg_e$ or $hg_e^{-1}$.
Let $\alpha$ be a loop in $X$ starting at a vertex of $V$.
We denote
$\alpha=\{v_i,\{v_i,v_{i+1}\}\mid1\leq{}i\leq{}k,v_{k+1}=v_1\}$.
Note that $v_1,g_1^{-1}v_2\in{}V$, where $g_1=g_{\{v_1,v_2\}}$.
For $2\leq{}i\leq{}k$, define
$g_i=g_{g_{i-1}^{-1}\cdots{}g_1^{-1}\{v_i,v_{i+1}\}}$, inductively. 
Note that for $2\leq{}i\leq{}k$, there exists an oriented $1$-simplex
$e_i$ such that $o(e_i)\in{}V$ and
$\{v_i,v_{i+1}\}=g_1g_2\cdots{}g_{i-1}\{o(e_i),t(e_i)\}$.
Let $g_{\alpha}=g_1g_2\cdots{}g_k$.
We have $g_{\alpha}(v_1)=v_1$, that is, $g_{\alpha}\in{}G_{v_1}$.

For $e\in{}E$, put a word $\hat{g}_e$.
For a $1$-simplex $\{v,w\}$ with $v\in{}V$, let
$\hat{g}_{\{v,w\}}=h\hat{g}_e$ or $h\hat{g}_e^{-1}$ if
$g_{\{v,w\}}=hg_e$ or $hg_e^{-1}$, respectively.
For a loop $\alpha$ in $X$ starting at a vertex of $V$, let 
$\hat{g}_{\alpha}=\hat{g}_1\hat{g}_2\cdots\hat{g}_k$ if
$g_{\alpha}=g_1g_2\cdots{}g_k$.
Note that we can define $g_{\tau}$ and $\hat{g}_{\tau}$ for
$\tau\in{}F$, regarding $\tau$ as a loop in $X$.
Let
$\widehat{G}=\left(\underset{v\in{}V}{\ast}G_v\right)\ast\left(\underset{e\in{}E}{\ast}\la\hat{g}_e\ra\right)$.

The following theorem is a special case of the result of Brown
\cite{b}.

\begin{thm}[\cite{b}]\label{Brown}
Let $X$ be a simply connected simplicial complex, and let $G$ be a group
 acting without rotation on $X$ by isomorphisms as a simplicial map.
Then $G$ is isomorphic to the quotient of $\widehat{G}$ by the normal
 subgroup generated by followings
\begin{enumerate}
 \item $\hat{g}_e$, where $e\in{}T$,
 \item $\hat{g}_e^{-1}A_{o(e)}\hat{g}_e(g_e^{-1}Ag_e)_{w(e)}^{-1}$,
       where $e\in{}E$ and $A\in{}G_e$,
 \item $\hat{g}_{\tau}g_{\tau}^{-1}$, where $\tau\in{}F$.
\end{enumerate}
\end{thm}

\section{Proof of the case $n=2$ of Theorem~\ref{thm}}\label{n=2}

In this section, we prove the following proposition.

\begin{prop}\label{g(2)}
$\g(2)$ has a finite presentation with generators 
$E_{12}$,
$E_{21}$,
$F_1$ and $F_2$, and with relators
$F_1^2$,
$F_2^2$,
$(E_{12}F_1)^2$,
$(E_{12}F_2)^2$,
$(E_{21}F_1)^2$,
$(E_{21}F_2)^2$ and $(F_1F_2)^2$.
\end{prop}

\subsection{The Reidemeister Schreier method}\label{R-S}\

Let $x,y$ and $z$ be
$$
\begin{array}{ccc}
x=\left(
\begin{array}{rr}
1&-1\\
0&1\\
\end{array}
\right),&
y=\left(
\begin{array}{rr}
1&0\\
1&1\\
\end{array}
\right),&
z=\left(
\begin{array}{rr}
0&1\\
1&0\\
\end{array}
\right).
\end{array}
$$

At first, we prove the next lemma.

\begin{lem}\label{GL2Z}
$GL(2;\Z)$ has a presentation with
$$GL(2;\Z)=\la{}x,y,z\mid{}xyxy^{-1}x^{-1}y^{-1},(xy)^6,z^2,xzyz\ra.$$
\end{lem}

\begin{proof}
In \cite{s}, it is known that $SL(2;\Z)$ has a presentation with
$$SL(2;\Z)=\la{}x,y\mid{}xyxy^{-1}x^{-1}y^{-1},(xy)^6\ra.$$
Consider the short exact sequence
$$
1\to{}SL(2;\Z)\to{}GL(2;\Z)\to\{\pm1\}\to1.
$$
Note that $\{\pm1\}=\la\det{}z\mid(\det{}z)^2\ra$.
Then we have that $GL(2;\Z)$ has a presentation with generators $x,y$
 and $z$, and with the following relations
\begin{itemize}
 \item $xyxy^{-1}x^{-1}y^{-1}=1$,
       $(xy)^6=1$,
 \item $z^2=1$,
 \item $zxz^{-1}=y^{-1}$,
       $zyz^{-1}=x^{-1}$.
\end{itemize}
Since $z^2=1$, we have $zxzy=1$ and $zyzx=1$.
Moreover the equation $zxzy=zyzx=1$ is obtained from $xzyz=1$.
Therefore, we obtain the claim.
\end{proof}

Next we consider the short exact sequence
$$
1\to\g(2)\to{}GL(2;\Z)\stackrel{\pi}{\to}GL(2;\Z_2)\to1.
$$
For $0\leq{}i\leq5$, let $a_i\in{}GL(2;\Z)$ be
$$
\begin{array}{ccc}
a_0=\left(
\begin{array}{rr}
1&0\\
0&1\\
\end{array}
\right),&
a_1=\left(
\begin{array}{rr}
1&1\\
0&1\\
\end{array}
\right),&
a_2=\left(
\begin{array}{rr}
1&0\\
1&1\\
\end{array}
\right),\\
a_3=\left(
\begin{array}{rr}
0&1\\
1&0\\
\end{array}
\right),&
a_4=\left(
\begin{array}{rr}
1&1\\
1&0\\
\end{array}
\right),&
a_5=\left(
\begin{array}{rr}
0&1\\
1&1\\
\end{array}
\right),
\end{array}
$$
and let $U=\{a_0,a_1,a_2,a_3,a_4,a_5\}$.
Since each of $a_i$ is denoted by
$a_0=1$,
$a_1=x^{-1}$,
$a_2=y$,
$a_3=z$,
$a_4=x^{-1}z$ and
$a_5=yz$,
as a word over $\{x,y,z\}$, we have that $U$ is a Schreier transversal
for $\g(2)$ in $GL(2;\Z)$ (see \cite{j}).
For $A\in{}GL(2;\Z)$, we define $\overline{A}=a_i$ if
$\pi(A)=\pi(a_i)$.
Let $B$ be the set of matrices $\overline{wa_i}^{-1}wa_i$ with
$wa_i\notin{}U$, where $0\leq{}i\leq5$ and $w=x^{\pm1}$, $y^{\pm1}$ and
$z$.
Then we have 
$$
B=\left\{
\left(
\begin{array}{rr}
1&2\\
0&1\\
\end{array}
\right),
\left(
\begin{array}{rr}
1&-2\\
0&1\\
\end{array}
\right),
\left(
\begin{array}{rr}
1&2\\
0&-1\\
\end{array}
\right),
\left(
\begin{array}{rr}
1&0\\
2&1\\
\end{array}
\right),
\left(
\begin{array}{rr}
1&0\\
-2&1\\
\end{array}
\right),
\left(
\begin{array}{rr}
-1&0\\
2&1\\
\end{array}
\right)
\right\}
$$
(see Table~\ref{wa}).
Note that $B$ is a generating set of $\g(2)$ (see \cite{j}).
It is clear that
$$
\begin{array}{cc}
\left(
\begin{array}{rr}
1&-2\\
0&1\\
\end{array}
\right)
=\left(
\begin{array}{rr}
1&2\\
0&1\\
\end{array}
\right)^{-1},&
\left(
\begin{array}{rr}
1&0\\
-2&1\\
\end{array}
\right)
=\left(
\begin{array}{rr}
1&0\\
2&1\\
\end{array}
\right)^{-1}.
\end{array}
$$
Thus, by Tietze transformations, we obtain the generating set
$B'=\{g_1,g_2,g_3,g_4\}$ of $\g(2)$, where
$$
\begin{array}{cccc}
g_1=\left(
\begin{array}{rr}
1&2\\
0&1\\
\end{array}
\right),&
g_2=\left(
\begin{array}{rr}
1&2\\
0&-1\\
\end{array}
\right),&
g_3=\left(
\begin{array}{rr}
1&0\\
2&1\\
\end{array}
\right),&
g_4=\left(
\begin{array}{rr}
-1&0\\
2&1\\
\end{array}
\right).
\end{array}
$$

\begin{table}[htbp]
\begin{tabular}{|c|c|c|c|c|c|}
\hline
$\overline{wa_i}^{-1}wa_i$
&$w=x$&$w=x^{-1}$&$w=y$&$w=y^{-1}$&$w=z$\\\hline
$i=0$
&
$\left(
\begin{array}{rr}
1&-2\\
0&1\\
\end{array}
\right)$
&
$\left(
\begin{array}{rr}
1&0\\
0&1\\
\end{array}
\right)$
&
$\left(
\begin{array}{rr}
1&0\\
0&1\\
\end{array}
\right)$
&
$\left(
\begin{array}{rr}
1&0\\
-2&1\\
\end{array}
\right)$
&
$\left(
\begin{array}{rr}
1&0\\
0&1\\
\end{array}
\right)$
\\\hline
$i=1$
&
$\left(
\begin{array}{rr}
1&0\\
0&1\\
\end{array}
\right)$
&
$\left(
\begin{array}{rr}
1&2\\
0&1\\
\end{array}
\right)$
&
$\left(
\begin{array}{rr}
1&2\\
0&-1\\
\end{array}
\right)$
&
$\left(
\begin{array}{rr}
-1&0\\
2&1\\
\end{array}
\right)$
&
$\left(
\begin{array}{rr}
1&0\\
0&1\\
\end{array}
\right)$
\\\hline
$i=2$
&
$\left(
\begin{array}{rr}
1&2\\
0&-1\\
\end{array}
\right)$
&
$\left(
\begin{array}{rr}
-1&0\\
2&1\\
\end{array}
\right)$
&
$\left(
\begin{array}{rr}
1&0\\
2&1\\
\end{array}
\right)$
&
$\left(
\begin{array}{rr}
1&0\\
0&1\\
\end{array}
\right)$
&
$\left(
\begin{array}{rr}
1&0\\
0&1\\
\end{array}
\right)$
\\\hline
$i=3$
&
$\left(
\begin{array}{rr}
1&0\\
-2&1\\
\end{array}
\right)$
&
$\left(
\begin{array}{rr}
1&0\\
0&1\\
\end{array}
\right)$
&
$\left(
\begin{array}{rr}
1&0\\
0&1\\
\end{array}
\right)$
&
$\left(
\begin{array}{rr}
1&-2\\
0&1\\
\end{array}
\right)$
&
$\left(
\begin{array}{rr}
1&0\\
0&1\\
\end{array}
\right)$
\\\hline
$i=4$
&
$\left(
\begin{array}{rr}
1&0\\
0&1\\
\end{array}
\right)$
&
$\left(
\begin{array}{rr}
1&0\\
2&1\\
\end{array}
\right)$
&
$\left(
\begin{array}{rr}
-1&0\\
2&1\\
\end{array}
\right)$
&
$\left(
\begin{array}{rr}
1&2\\
0&-1\\
\end{array}
\right)$
&
$\left(
\begin{array}{rr}
1&0\\
0&1\\
\end{array}
\right)$
\\\hline
$i=5$
&
$\left(
\begin{array}{rr}
-1&0\\
2&1\\
\end{array}
\right)$
&
$\left(
\begin{array}{rr}
1&2\\
0&-1\\
\end{array}
\right)$
&
$\left(
\begin{array}{rr}
1&2\\
0&1\\
\end{array}
\right)$
&
$\left(
\begin{array}{rr}
1&0\\
0&1\\
\end{array}
\right)$
&
$\left(
\begin{array}{rr}
1&0\\
0&1\\
\end{array}
\right)$
\\\hline
\end{tabular}
\caption{The matrix $\overline{wa_i}^{-1}wa_i$.}\label{wa}
\end{table} 

We now prove the next lemma.

\begin{lem}\label{r}
Let $r=r_1r_2\cdots{}r_n\in{}GL(2;\Z)$.
Then for $0\leq{}i\leq5$ and $1\leq{}j\leq{}n-1$, we have
$$\overline{r_j\overline{(r_{j+1}\cdots{}r_n)a_i}}=\overline{(r_jr_{j+1}\cdots{}r_n)a_i}.$$
\end{lem}

\begin{proof}
Note that $\overline{A}=\overline{B}$ if and only if $\pi(A)=\pi(B)$.
We calculate
\begin{eqnarray*}
\pi(r_j\overline{(r_{j+1}\cdots{}r_n)a_i})
&=&
\pi(r_j)\pi(\overline{(r_{j+1}\cdots{}r_n)a_i})\\
&=&
\pi(r_j)\pi((r_{j+1}\cdots{}r_n)a_i)\\
&=&
\pi((r_jr_{j+1}\cdots{}r_n)a_n).
\end{eqnarray*} 
Therefore, we obtain the claim.
\end{proof}

Let $R$ be the set of relators of $GL(2;\Z)$ in Lemma~\ref{GL2Z}.
For any $r=r_1r_2\cdots{}r_n\in{}R$ and $0\leq{}i\leq5$, we define a
word $s_{ri}$ over $B'$ as follows.
$$
s_{ri}=
(a_i^{-1}r_1\overline{(r_2\cdots{}r_n)a_i})
(\overline{(r_2\cdots{}r_n)a_i}^{-1}r_2\overline{(r_3\cdots{}r_n)a_i})
\cdots(\overline{r_na_i}^{-1}r_na_i).
$$
Let $\widehat{S}=\{s_{ri}\mid{}r\in{}R,0\leq{}i\leq5\}$.
Then $\widehat{S}$ is a set of relators of $\g(2)$ (see \cite{j}).
Hence we have $\g(2)=\la{}B'\mid\widehat{S}\ra$.

\subsection{Proof of Proposition \ref{g(2)}}\

We now write all elements in $\widehat{S}$ as a product of elements in
$B'$. 
Let $[w]=\overline{w}^{-1}w$.

For $r=xyxy^{-1}x^{-1}y^{-1}$, we have
\begin{eqnarray*}
s_{r0}
&=&
[xa_1][ya_4][xa_3][y^{-1}a_5][x^{-1}a_2][y^{-1}a_0]\\
&=&
(g_4g_3^{-1})^2,\\
s_{r1}
&=&
[xa_0][ya_2][xa_5][y^{-1}a_3][x^{-1}a_4][y^{-1}a_1]\\
&=&
(g_1^{-1}g_3g_4)^2,\\
s_{r2}
&=&
[xa_5][ya_3][xa_4][y^{-1}a_1][x^{-1}a_0][y^{-1}a_2]\\
&=&
g_4^2,\\
s_{r3}
&=&
[xa_4][ya_1][xa_0][y^{-1}a_2][x^{-1}a_5][y^{-1}a_3]\\
&=&
(g_2g_1^{-1})^2,\\
s_{r4}
&=&
[xa_3][ya_5][xa_2][y^{-1}a_0][x^{-1}a_1][y^{-1}a_4]\\
&=&
(g_3^{-1}g_1g_2)^2,\\
s_{r5}
&=&
[xa_2][ya_0][xa_1][y^{-1}a_4][x^{-1}a_3][y^{-1}a_5]\\
&=&
g_2^2.
\end{eqnarray*} 
For $r=(xy)^6$, we have
\begin{eqnarray*}
s_{r0}
&=&
[xa_1][ya_4][xa_3][ya_5][xa_2][ya_0][xa_1][ya_4][xa_3][ya_5][xa_2][ya_0]\\
&=&
(g_4g_3^{-1}g_1g_2)^2,\\
s_{r1}
&=&
[xa_0][ya_2][xa_5][ya_3][xa_4][ya_1][xa_0][ya_2][xa_5][ya_3][xa_4][ya_1]\\
&=&
(g_1^{-1}g_3g_4g_2)^2,\\
s_{r2}
&=&
[xa_5][ya_3][xa_4][ya_1][xa_0][ya_2][xa_5][ya_3][xa_4][ya_1][xa_0][ya_2]\\
&=&
(g_4g_2g_1^{-1}g_3)^2,\\
s_{r3}
&=&
[xa_4][ya_1][xa_0][ya_2][xa_5][ya_3][xa_4][ya_1][xa_0][ya_2][xa_5][ya_3]\\
&=&
(g_2g_1^{-1}g_3g_4)^2,\\
s_{r4}
&=&
[xa_3][ya_5][xa_2][ya_0][xa_1][ya_4][xa_3][ya_5][xa_2][ya_0][xa_1][ya_4]\\
&=&
(g_3^{-1}g_1g_2g_4)^2,\\
s_{r5}
&=&
[xa_2][ya_0][xa_1][ya_4][xa_3][ya_5][xa_2][ya_0][xa_1][ya_4][xa_3][ya_5]\\
&=&
(g_2g_4g_3^{-1}g_1)^2.
\end{eqnarray*} 
For $r=z^2$ and $0\leq{}i\leq5$, since
$\overline{za_i}^{-1}za_i=
\left(
\begin{array}{rr}
1&0\\
0&1\\
\end{array}
\right)$,
we have $s_{ri}=1$.
For $r=xzyz$, we have
\begin{eqnarray*}
s_{r0}
&=&
[xa_1][za_5][ya_3][za_0]=1,\\
s_{r1}
&=&
[xa_0][za_3][ya_5][za_1]=g_1^{-1}g_1=1,\\
s_{r2}
&=&
[xa_5][za_1][ya_4][za_2]=g_4^2,\\
s_{r3}
&=&
[xa_4][za_2][ya_0][za_3]=1,\\
s_{r4}
&=&
[xa_3][za_0][ya_2][za_4]=g_3^{-1}g_3=1,\\
s_{r5}
&=&
[xa_2][za_4][ya_1][za_5]=g_2^2.
\end{eqnarray*} 

Note that 
$s_{(xy)^60}=s_{(xy)^64}=s_{(xy)^65}$,
$s_{(xy)^61}=s_{(xy)^62}=s_{(xy)^63}$,
up to conjugation, and 
$s_{xzyz2}=s_{xyxy^{-1}x^{-1}y^{-1}2}$,
$s_{xzyz5}=s_{xyxy^{-1}x^{-1}y^{-1}5}$.
Therefore, $\g(2)$ has a presentation with generators $g_1,g_2,g_3,g_4$
and with relators
$(g_4g_3^{-1})^2$,
$(g_1^{-1}g_3g_4)^2$,
$g_4^2$,
$(g_2g_1^{-1})^2$,
$(g_3^{-1}g_1g_2)^2$,
$g_2^2$,
$(g_4g_3^{-1}g_1g_2)^2$ and
$(g_1^{-1}g_3g_4g_2)^2$.

Finally, we put
$E_{12}=g_1$,
$E_{21}=g_3$,
$F_1=g_4g_3^{-1}$ and $F_2=g_2g_1^{-1}$.
Note that 
$g_1=E_{12}$,
$g_2=F_2E_{12}$,
$g_3=E_{21}$ and $g_4=F_1E_{21}$.
By Tietze transformations, we conclude that $\g(2)$ has a finite
presentation with generators
$E_{12}$,
$E_{21}$,
$F_1$ and $F_2$, and with relators
$F_1^2$,
$F_2^2$,
$(E_{12}F_1)^2$,
$(E_{12}F_2)^2$,
$(E_{21}F_1)^2$,
$(E_{21}F_2)^2$ and $(F_1F_2)^2$.

Thus, the proof of Proposition~\ref{g(2)} is completed.
Therefore, Theorem~\ref{thm} is valid when $n=2$.

\section{Proof of the case $n=3$ of Theorem~\ref{thm}}\label{n=3}

In this section, we prove the following proposition.

\begin{prop}\label{g(3)}
$\g(3)$ has a finite presentation with generators
$E_{12}$,
$E_{13}$,
$E_{21}$,
$E_{23}$,
$E_{31}$,
$E_{32}$,
$F_1$,
$F_2$ and 
$F_2$,
and with the following relators
\begin{enumerate}
 \item $F_1^2$,
       $F_2^2$,
       $F_3^2$,
 \item $(E_{12}F_1)^2$,
       $(E_{12}F_2)^2$,
       $(E_{13}F_1)^2$,
       $(E_{13}F_3)^2$,
       $(E_{21}F_2)^2$,
       $(E_{21}F_1)^2$,
       $(E_{23}F_2)^2$,
       $(E_{23}F_3)^2$,
       $(E_{31}F_3)^2$,
       $(E_{31}F_1)^2$,
       $(E_{32}F_3)^2$,
       $(E_{32}F_2)^2$,
       $(F_1F_2)^2$,
       $(F_1F_3)^2$,
       $(F_2F_3)^2$,
 \item \begin{enumerate}
	\item $[E_{12},E_{13}]$,
	      $[E_{21},E_{23}]$,
	      $[E_{31},E_{32}]$,
	      $[E_{21},E_{31}]$,
	      $[E_{12},E_{32}]$,
	      $[E_{13},E_{23}]$,
	      $[E_{12},F_3]$,
	      $[E_{21},F_3]$,
	      $[E_{13},F_2]$,
	      $[E_{31},F_2]$,
	      $[E_{23},F_1]$,
	      $[E_{32},F_1]$,
	      $[E_{32},E_{13}]E_{12}^2$,
	      $[E_{23},E_{12}]E_{13}^2$,
	      $[E_{31},E_{23}]E_{21}^2$,
	      $[E_{13},E_{21}]E_{23}^2$,
	      $[E_{21},E_{32}]E_{31}^2$,
	      $[E_{12},E_{31}]E_{32}^2$,
	\item $[E_{21}F_2E_{12}F_1E_{31}^{-1}E_{32},E_{31}F_3E_{13}F_1E_{21}^{-1}E_{23}]$.
       \end{enumerate}
\end{enumerate}
\end{prop}

\subsection{Preparation}\

For $R=\Z$ or $\Z_2$, let $\B_n(R)$ denote the simplicial complex whose
$(k-1)$-simplex $\{x_1,x_2,\dots,x_k\}$ is the set of $k$-vectors
$x_i\in{}R^n$ such that $x_1,x_2,\dots,x_k$ are mutually different
column vectors of a matrix $A\in{}GL(n;R)$.
In \cite{dp}, Day and Putman proved that $\B_n(\Z)$ is
$(n-2)$-connected.
Here, a simplicial complex $X$ is $m$-connected if its geometric
realization $|X|$ is $m$-connected.
In addition, $X$ is $-1$-connected if $X$ is nonempty.
Note that there is the natural left action
$\g(n)\times\B_n(\Z)\to\B_n(\Z)$ defined by
$A\{x_1,x_2,\dots,x_k\}=\{Ax_1,Ax_2,\dots,Ax_k\}$ for $A\in\g(n)$ and
$\{x_1,x_2,\dots,x_k\}\in\B_n(\Z)$, and that the action is without
rotation.

In this section, we consider the case $n=3$.
Since $GL(3;\Z_2)$ is the quotient of $GL(3;\Z)$ by $\g(3)$, it follows
that the orbit space $\g(3)\backslash\B_3(\Z)$ is isomorphic to
$\B_3(\Z_2)$.
Let $\varphi:\B_3(\Z)\to\B_3(\Z_2)$ be a natural surjection induced by
the surjection $GL(3;\Z)\twoheadrightarrow{}GL(3;\Z_2)$.

For $1\leq{}i\leq7$, let $v_i$ be
$v_1=e_1$,
$v_2=e_2$,
$v_3=e_3$,
$v_4=e_1+e_2$,
$v_5=e_1+e_3$,
$v_6=e_2+e_3$ and 
$v_7=e_1+e_2+e_3$,
where $e_1$, $e_2$ and $e_3$ are canonical normal vectors in $\Z^3$.
Then, the vertices of $\B_3(\Z_2)$ are $\varphi(v_i)$, the $1$-simplices
are $\varphi(\{v_i,v_j\})$, and the $2$-simplices are
$\varphi(\{v_i,v_j,v_k\})$, where $\{i,j,k\}$ is not 
$\{1,2,4\}$,
$\{1,3,5\}$,
$\{1,6,7\}$,
$\{2,3,6\}$,
$\{2,5,7\}$,
$\{3,4,7\}$ and
$\{4,5,6\}$.
(Note that
$\{v_1,v_2,v_4\}$,
$\{v_1,v_3,v_5\}$,
$\{v_1,v_6,v_7\}$,
$\{v_2,v_3,v_6\}$,
$\{v_2,v_5,v_7\}$,
$\{v_3,v_4,v_7\}$ and
$\{v_4,v_5,v_6\}$
are not $2$-simplices of $\B_3(\Z)$.)

We prove the next lemma.

\begin{lem}\label{g(3)2}
 $\g(3)$ is isomorphic to the quotient of
 $\underset{1\leq{}i\leq7}{\ast}\g(3)_{v_i}$ by the normal subgroup
 generated by edge relators.
\end{lem}

For the definition of the edge relator, see Subsection~\ref{brown}.

\begin{proof}
We set followings
\begin{itemize}
 \item $V=\{v_1,v_2,v_3,v_4,v_5,v_6,v_7\}$,
 \item $T=\{(v_1,v_i)\mid2\leq{}i\leq7\}\cup{}V$,
 \item $E=\{(v_i,v_j)\mid1\leq{}i<j\leq7\}$,
 \item $F=\{(v_i,v_j,v_k)\mid1\leq{}i<j<k\leq7,\varphi(\{v_i,v_j,v_k\})\in\B_3(\Z_2)\}$. 
\end{itemize}
For $e=(v_i,v_j)\in{}E$, since $w(e)=t(e)$, we choose $g_e=1$, and write
 $g_{ij}=g_e$.
By Theorem~\ref{Brown}, $\g(3)$ is isomorphic to the quotient of
$\left(\underset{1\leq{}i\leq7}{\ast}\g(3)_{v_i}\right)\ast\left(\underset{1\leq{}i<j\leq7}{\ast}\la\hat{g}_{ij}\ra\right)$
by the normal subgroup generated by followings
\begin{enumerate}
 \item $\hat{g}_{1i}$, where $2\leq{}i\leq7$,
 \item $\hat{g}_{ij}^{-1}X_{v_i}\hat{g}_{ij}X_{v_j}^{-1}$, where
       $1\leq{}i<j\leq7$ and $X\in\g(3)_{(v_i,v_j)}$,
 \item $\hat{g}_{\tau}g_{\tau}^{-1}$, where $\tau\in{}F$.
\end{enumerate}
Note that $g_{\tau}=g_{ij}g_{jk}g_{ik}^{-1}$ for $\tau=(v_i,v_j,v_k)$.
Hence, the relation $\hat{g_{\tau}}g_{\tau}^{-1}=1$ is equivalent to the
relation $\hat{g}_{ij}\hat{g}_{jk}=\hat{g}_{ik}$.
Since $\hat{g}_{1i}=1$ for $2\leq{}i\leq7$, we have the relation
 $\hat{g}_{ij}=1$ for $2\leq{}i<j\leq7$ except for $(i,j)=(2,4)$,
 $(3,5)$ and $(6,7)$.
For example, the relation $\hat{g}_{23}=1$ is obtained from the relation
$\hat{g}_{12}\hat{g}_{23}=\hat{g}_{13}$.
In addition, relations 
$\hat{g}_{24}=1$,
$\hat{g}_{35}=1$ and
$\hat{g}_{67}=1$
are obtained from relations
$\hat{g}_{23}\hat{g}_{34}=\hat{g}_{24}$,
$\hat{g}_{23}\hat{g}_{35}=\hat{g}_{25}$ and
$\hat{g}_{26}\hat{g}_{67}=\hat{g}_{27}$,
respectively.
Hence, we have the relation $\hat{g}_{ij}=1$ for $1\leq{}i<j\leq7$.
Therefore, $\g(3)$ is isomorphic to the quotient of
$\underset{1\leq{}i\leq7}{\ast}\g(3)_{v_i}$ by the normal subgroup generated by 
$A=\{X_{v_i}X_{v_j}^{-1}\mid1\leq{}i<j\leq7,X\in\g(3)_{(v_i,v_j)}\}$.
Since $A$ is the set of edge relators, we obtain the claim.
\end{proof}

We next consider presentations of $\g(3)_{v_i}$ for all $1\leq{}i\leq7$
and edge relators.

\subsection{Presentations of $\g(3)_{v_i}$}\label{g(3)i-pre}\

\begin{lem}\label{ses}
For $1\leq{}t\leq{}n$ there is a short exact sequence
$$0\to\Z^{n-1}\to\g(n)_{e_t}\to\g(n-1)\to1.$$
\end{lem}

\begin{proof}
We first note that $A\in\g(n)_{e_t}$ is a matrix whose $t$-column vector
is $e_t$.
For $\Z^{n-1}$ we give the presentation
$\Z^{n-1}=\la{}x_1,x_2,\dots,x_{n-1}\mid{}x_ix_jx_i^{-1}x_j^{-1}(1\leq{}i<j\leq{}n-1)\ra$.
Let $\Z^{n-1}\to\g(n)_{e_t}$ be the homomorphism which sends $x_i$ to
 $E_{ti}$ when $i<t$ and to $E_{ti+1}$ when $i\geq{}t$.
Let $\g(n)_{e_t}\to\g(n-1)$ be the homomorphism which sends $A$ to
 $A_{tt}$, where $A_{ij}$ is the $(n-1)$-submatrix of $A$ obtained by
 removing the $i$-row vector and the $j$-column vector of $A$.
Then, it follows that the sequence
 $0\to\Z^{n-1}\to\g(n)_{e_t}\to\g(n-1)\to1$ is exact.
\end{proof}

\begin{rem}\label{sesrem}
Let $\rho_t:\g(n-1)\to\g(n)_{e_t}$ be the homomorphism defined by
\begin{eqnarray*}
\rho_t(E_{ij})&=&\left\{
\begin{array}{ll}
(E_{ij})_{e_t}&({\rm when}~i,j\leq t-1),\\
(E_{ij+1})_{e_t}&({\rm when}~i\leq t-1,~j\geq{}t),\\
(E_{i+1j})_{e_t}&({\rm when}~j\leq t-1,~i\geq{}t),\\
(E_{i+1j+1})_{e_t}&({\rm when}~i,j\geq{}t),
\end{array}
\right.\\
\rho_t(F_i)&=&\left\{
\begin{array}{ll}
(F_i)_{e_t}&({\rm when}~i\leq{}t-1),\\
(F_{i+1})_{e_t}&({\rm when}~i\geq{}t),
\end{array}
\right.
\end{eqnarray*}
where subscripts $e_t$ are added in order to indicate that these are
 the elements of $\g(n)_{e_t}$, that is, we write $A_{e_t}$ for
 $A\in\g(n)_{e_t}$.
Put $\g(n-1)=\la{}X\mid{}Y\ra$.
Then, from Lemma \ref{ses}, $\g(n)_{e_t}$ is generated by
\begin{itemize}
 \item $(E_{ti})_{e_t}$ for $1\leq{}i\leq{}n$ with $i\neq{}t$,
 \item $(E_{ij})_{e_t}$, $(F_i)_{e_t}$ for $1\leq{}i,j\leq{}n$ with
       $i\neq{}j$ and $i,j\neq{}t$,
\end{itemize}
and has relators
\begin{enumerate}
 \item $[(E_{ti})_{e_t},(E_{tj})_{e_t}]$ for $1\leq{}i,j\leq{}n$ with
       $i\neq{}j$,
 \item $\rho_t(y)$ for $y\in{}Y$,
 \item \begin{itemize}
	\item $(E_{ij})_{e_t}^{-1}(E_{ti})_{e_t}(E_{ij})_{e_t}\cdot(E_{tj})_{e_t}^{-2}(E_{ti})_{e_t}^{-1}$
	      for $1\leq{}i,j\leq{}n$ with $i\neq{}j$ and $i,j\neq{}t$,
	\item $(E_{ij})_{e_t}^{-1}(E_{tj})_{e_t}(E_{ij})_{e_t}\cdot(E_{tj})_{e_t}^{-1}$
	      for $1\leq{}i,j\leq{}n$ with $i\neq{}j$ and $i,j\neq{}t$,
	\item $(E_{ij})_{e_t}^{-1}(E_{tk})_{e_t}(E_{ij})_{e_t}\cdot(E_{tk})_{e_t}^{-1}$
	      for $1\leq{}i,j,k\leq{}n$ with $i,j,k\neq{}t$ and $i,j,k$
	      are mutually different (when $n\geq4$),
	\item $(F_i)_{e_t}^{-1}(E_{ti})_{e_t}(F_i)_{e_t}\cdot(E_{ti})_{e_t}$
	      for $1\leq{}i\leq{}n$ with $i\neq{}t$,
	\item $(F_i)_{e_t}^{-1}(E_{tj})_{e_t}(F_i)_{e_t}\cdot(E_{tj})_{e_t}^{-1}$
	      for $1\leq{}i,j\leq{}n$ with $i\neq{}j$ and $i,j\neq{}t$.
       \end{itemize}
\end{enumerate}
The relators (3) can be rephrased as follows.
\begin{itemize}
 \item $[(E_{ij})_{e_t},(E_{ti})_{e_t}](E_{tj})_{e_t}^2$
       for $1\leq{}i,j\leq{}n$ with $i\neq{}j$ and $i,j\neq{}t$,
 \item $[(E_{ij})_{e_t},(E_{tj})_{e_t}]$
       for $1\leq{}i,j\leq{}n$ with $i\neq{}j$ and $i,j\neq{}t$,
 \item $[(E_{ij})_{e_t},(E_{tk})_{e_t}]$
       for $1\leq{}i,j,k\leq{}n$ with $i,j,k\neq{}t$ and $i,j,k$ are
       mutually different (when $n\geq4$),
 \item $((E_{ti})_{e_t}(F_i)_{e_t})^2$
       for $1\leq{}i\leq{}n$ with $i\neq{}t$,
 \item $[(E_{tj})_{e_t},(F_i)_{e_t}]$
       for $1\leq{}i,j\leq{}n$ with $i\neq{}j$ and $i,j\neq{}t$.
\end{itemize}
\end{rem}

By Lemma~\ref{ses}, Remark~\ref{sesrem} and Proposition~\ref{g(2)}, we
have the following.

\begin{lem}\label{g(3)_i}
$\g(3)_{v_1}$ has a finite presentation with generators 
$\aaa$,
$\aab$,
$\aac$,
$\aad$,
$\aae$ and
$\aaf$, and with the following relators
\begin{enumerate}
 \item[(1.1)] $(\aae)^2$,
	      $(\aaf)^2$,
 \item[(1.2)] $(\aaa\aae)^2$,
	      $(\aab\aaf)^2$,
	      $(\aac\aae)^2$,
	      $(\aac\aaf)^2$,\\
	      $(\aad\aae)^2$,
	      $(\aad\aaf)^2$,
	      $(\aae\aaf)^2$,
 \item[(1.3)] $[\aaa,\aab]$,
	      $[\aaa,\aad]$,
	      $[\aaa,\aaf]$,
	      $[\aab,\aac]$,\\
	      $[\aab,\aae]$,
	      $[\aac,\aaa]\aab^2$,
	      $[\aad,\aab]\aaa^2$.
\end{enumerate}
\end{lem}

For $X\in{}GL(n;\Z)$, let $\Phi_X:\g(n)\to\g(n)$ be the homomorphism
defined by $\Phi_X(A)=XAX^{-1}$.
Note that this definition is well-defined, since $\g(n)$ is a normal
subgroup of $GL(n;\Z)$.
For $1\leq{}i,j\leq{}n$ with $i\neq{}j$, let $T_{ij}$ denote the matrix
whose $(i,j)$ entry is $1$, diagonal entries are $1$ and others are $0$,
and let $S_i$ denote the matrix whose $(i,i)$ and $(i+1,i+1)$ entries are
$0$, other diagonal entries are $1$, $(i,i+1)$ and $(i+1,i)$ entries are
$1$ and others are $0$.
Using homomorphisms $\Phi_X$ for some $X\in{}GL(n;\Z)$, we provide
presentations of $\g(n)_{v_i}$ for all $2\leq{}i\leq7$.

First, considering $\Phi_{S_1}:\g(3)_{v_1}\to\g(3)_{v_2}$, it follows
that
$\g(3)_{v_2}$ has a finite presentation with generators
$\bba$,
$\bbb$,
$\bbc$,
$\bbd$,
$\bbe$ and
$\bbf$, and with the following relators
\begin{enumerate}
 \item[(2.1)] $(\bbe)^2$,
	      $(\bbf)^2$,
 \item[(2.2)] $(\bba\bbe)^2$,
	      $(\bbb\bbf)^2$,
	      $(\bbc\bbe)^2$,
	      $(\bbc\bbf)^2$,\\
	      $(\bbd\bbe)^2$,
	      $(\bbd\bbf)^2$,
	      $(\bbe\bbf)^2$,
 \item[(2.3)] $[\bba,\bbb]$,
	      $[\bba,\bbd]$,
	      $[\bba,\bbf]$,
	      $[\bbb,\bbc]$,\\
	      $[\bbb,\bbe]$,
	      $[\bbc,\bba]\bbb^2$,
	      $[\bbd,\bbb]\bba^2$.
\end{enumerate}

Next, considering $\Phi_{S_2S_1}:\g(3)_{v_1}\to\g(3)_{v_3}$, it follows
that
$\g(3)_{v_3}$ has a finite presentation with generators
$\cca$,
$\ccb$,
$\ccc$,
$\ccd$,
$\cce$ and
$\ccf$, and with the following relators
\begin{enumerate}
 \item[(3.1)] $(\cce)^2$,
	      $(\ccf)^2$,
 \item[(3.2)] $(\cca\cce)^2$,
	      $(\ccb\ccf)^2$,
	      $(\ccc\cce)^2$,
	      $(\ccc\ccf)^2$,\\
	      $(\ccd\cce)^2$,
	      $(\ccd\ccf)^2$,
	      $(\cce\ccf)^2$,
 \item[(3.3)] $[\cca,\ccb]$,
	      $[\cca,\ccd]$,
	      $[\cca,\ccf]$,
	      $[\ccb,\ccc]$,\\
	      $[\ccb,\cce]$,
	      $[\ccc,\cca]\ccb^2$,
	      $[\ccd,\ccb]\cca^2$.
\end{enumerate}

Next, considering $\Phi_{T_{21}}:\g(3)_{v_1}\to\g(3)_{v_4}$, it follows
that
$\g(3)_{v_4}$ has a finite presentation with generators
$\dda$,
$\ddb$,
$\ddc$,
$\ddd$,
$\dde$ and
$\ddf$, and with the following relators
\begin{enumerate}
 \item[(4.1)] $(\dde)^2$,
	      $(\ddf)^2$,
 \item[(4.2)] $(\dda\dde)^2$,
	      $(\ddb\ddf)^2$,
	      $(\ddc\dde)^2$,\\
	      $(\ddc\ddf)^2$,
	      $(\ddd\dde)^2$,
	      $(\ddd\ddf)^2$,
	      $(\dde\ddf)^2$,
 \item[(4.3)] $[\dda,\ddb]$,
	      $[\dda,\ddd]$,\\
	      $[\dda,\ddf]$,
	      $[\ddb,\ddc]$,
	      $[\ddb,\dde]$,\\
	      $[\ddc,\dda]\ddb^2$,
	      $[\ddd,\ddb]\dda^2$.
\end{enumerate}

Next, considering $\Phi_{S_2T_{21}}:\g(3)_{v_1}\to\g(3)_{v_5}$, it
follows that
$\g(3)_{v_5}$ has a finite presentation with generators
$\eea$,
$\eeb$,
$\eec$,
$\eed$,
$\eee$ and
$\eef$, and with the following relators
\begin{enumerate}
 \item[(5.1)] $(\eee)^2$,
	      $(\eef)^2$,
 \item[(5.2)] $(\eea\eee)^2$,
	      $(\eeb\eef)^2$,
	      $(\eec\eee)^2$,\\
	      $(\eec\eef)^2$,
	      $(\eed\eee)^2$,
	      $(\eed\eef)^2$,
	      $(\eee\eef)^2$,
 \item[(5.3)] $[\eea,\eeb]$,
	      $[\eea,\eed]$,\\
	      $[\eea,\eef]$,
	      $[\eeb,\eec]$,
	      $[\eeb,\eee]$,\\
	      $[\eec,\eea]\eeb^2$,
	      $[\eed,\eeb]\eea^2$.
\end{enumerate}

Next, considering $\Phi_{S_1S_2T_{21}}:\g(3)_{v_1}\to\g(3)_{v_6}$, it
follows that
$\g(3)_{v_6}$ has a finite presentation with generators
$\ffa$,
$\ffb$,
$\ffc$,
$\ffd$,
$\ffe$ and
$\fff$, and with the following relators
\begin{enumerate}
 \item[(6.1)] $(\ffe)^2$,
	      $(\fff)^2$,
 \item[(6.2)] $(\ffa\ffe)^2$,
	      $(\ffb\fff)^2$,
	      $(\ffc\ffe)^2$,\\
	      $(\ffc\fff)^2$,
	      $(\ffd\ffe)^2$,
	      $(\ffd\fff)^2$,
	      $(\ffe\fff)^2$,
 \item[(6.3)] $[\ffa,\ffb]$,
	      $[\ffa,\ffd]$,\\
	      $[\ffa,\fff]$,
	      $[\ffb,\ffc]$,
	      $[\ffb,\ffe]$,\\
	      $[\ffc,\ffa]\ffb^2$,
	      $[\ffd,\ffb]\ffa^2$.
\end{enumerate}

Finally, considering $\Phi_{T_{31}T_{21}}:\g(3)_{v_1}\to\g(3)_{v_7}$, it
follows that
$\g(3)_{v_7}$ has a finite presentation with generators
$\gga$,
$\ggb$,
$\ggc$,
$\ggd$,
$\gge$ and
$\ggf$, and with the following relators
\begin{enumerate}
 \item[(7.1)] $(\gge)^2$,
	      $(\ggf)^2$,
 \item[(7.2)] $(\gga\gge)^2$,
	      $(\ggb\ggf)^2$,\\
	      $(\ggc\gge)^2$,
	      $(\ggc\ggf)^2$,
	      $(\ggd\gge)^2$,\\
	      $(\ggd\ggf)^2$,
	      $(\gge\ggf)^2$,
 \item[(7.3)] $[\gga,\ggb]$,\\
	      $[\gga,\ggd]$,
	      $[\gga,\ggf]$,\\
	      $[\ggb,\ggc]$,
	      $[\ggb,\gge]$,\\
	      $[\ggc,\gga]\ggb^2$,
	      $[\ggd,\ggb]\gga^2$.
\end{enumerate}

\subsection{On edge relations}\label{edge}\

Note that
$$\g(3)_{(v_1,v_2)}=\g(3)_{(v_1,v_4)}=\g(3)_{(v_2,v_4)},$$
$$\g(3)_{(v_1,v_3)}=\g(3)_{(v_1,v_5)}=\g(3)_{(v_3,v_5)},$$
$$\g(3)_{(v_2,v_3)}=\g(3)_{(v_2,v_6)}=\g(3)_{(v_3,v_6)},$$
$$\g(3)_{(v_1,v_6)}=\g(3)_{(v_1,v_7)}=\g(3)_{(v_6,v_7)},$$
$$\g(3)_{(v_2,v_5)}=\g(3)_{(v_2,v_7)}=\g(3)_{(v_5,v_7)},$$
$$\g(3)_{(v_3,v_4)}=\g(3)_{(v_3,v_7)}=\g(3)_{(v_4,v_7)}.$$

It follows that
$\g(3)_{(v_1,v_2)}$,
$\g(3)_{(v_1,v_4)}$ and
$\g(3)_{(v_2,v_4)}$
are generated by
\begin{eqnarray*}
\left(
\begin{array}{rrr}
1&0&2\\
0&1&0\\
0&0&1
\end{array}
\right),&
\left(
\begin{array}{rrr}
1&0&0\\
0&1&2\\
0&0&1
\end{array}
\right),&
\left(
\begin{array}{rrr}
1&0&0\\
0&1&0\\
0&0&-1
\end{array}
\right).
\end{eqnarray*}
Then we have the following edge relations
\begin{itemize}
 \item $\aab=\bbc=\ddb\ddc^{-1}$,
 \item $\aac=\bbb=\ddc$,
 \item $\aaf=\bbf=\ddf$.
\end{itemize}

Next, considering $\Phi_{S_2}:\g(3)_{(v_1,v_2)}\to\g(3)_{(v_1,v_3)}$, it
follows that
$\g(3)_{(v_1,v_3)}$,
$\g(3)_{(v_1,v_5)}$ and
$\g(3)_{(v_3,v_5)}$
are generated by
\begin{eqnarray*}
\left(
\begin{array}{rrr}
1&2&0\\
0&1&0\\
0&0&1
\end{array}
\right),&
\left(
\begin{array}{rrr}
1&0&0\\
0&1&0\\
0&2&1
\end{array}
\right),&
\left(
\begin{array}{rrr}
1&0&0\\
0&-1&0\\
0&0&1
\end{array}
\right).
\end{eqnarray*}
Then we have the following edge relations
\begin{itemize}
 \item $\aaa=\ccc=\eeb\eec^{-1}$,
 \item $\aad=\ccb=\eec$,
 \item $\aae=\ccf=\eef$.
\end{itemize}

Next, considering $\Phi_{S_1S_2}:\g(3)_{(v_1,v_2)}\to\g(3)_{(v_2,v_3)}$,
it follows that
$\g(3)_{(v_2,v_3)}$,
$\g(3)_{(v_2,v_6)}$ and
$\g(3)_{(v_3,v_6)}$
are generated by
\begin{eqnarray*}
\left(
\begin{array}{rrr}
1&0&0\\
2&1&0\\
0&0&1
\end{array}
\right),&
\left(
\begin{array}{rrr}
1&0&0\\
0&1&0\\
2&0&1
\end{array}
\right),&
\left(
\begin{array}{rrr}
-1&0&0\\
0&1&0\\
0&0&1
\end{array}
\right).
\end{eqnarray*}
Then we have the following edge relations
\begin{itemize}
 \item $\bba=\ccd=\ffb\ffc^{-1}$,
 \item $\bbd=\cca=\ffc$,
 \item $\bbe=\cce=\fff$.
\end{itemize}

Next, considering $\Phi_{T_{32}}:\g(3)_{(v_1,v_2)}\to\g(3)_{(v_1,v_6)}$,
it follows that
$\g(3)_{(v_1,v_6)}$,
$\g(3)_{(v_1,v_7)}$ and
$\g(3)_{(v_6,v_7)}$
are generated by
\begin{eqnarray*}
\left(
\begin{array}{rrr}
1&-2&2\\
0&1&0\\
0&0&1
\end{array}
\right),&
\left(
\begin{array}{rrr}
1&0&0\\
0&-1&2\\
0&-2&3
\end{array}
\right),&
\left(
\begin{array}{rrr}
1&0&0\\
0&1&0\\
0&2&-1
\end{array}
\right).
\end{eqnarray*}
Then we have the following edge relations
\begin{itemize}
 \item $\aaa^{-1}\aab=\ffd\\
       =\ggf\ggb\ggc^{-1}\gge\\
       \cdot\gga\ggd^{-1}$,
 \item $\aad\aaf\aac\aae=\ffa\\
       =\ggd\ggf\ggc\gge$,
 \item $\aad\aaf=\ffe=\ggd\ggf$.
\end{itemize}

Next, considering
$\Phi_{S_1T_{32}}:\g(3)_{(v_1,v_2)}\to\g(3)_{(v_2,v_5)}$, it follows
that
$\g(3)_{(v_2,v_5)}$,
$\g(3)_{(v_2,v_7)}$ and
$\g(3)_{(v_5,v_6)}$
are generated by
\begin{eqnarray*}
\left(
\begin{array}{rrr}
1&0&0\\
-2&1&2\\
0&0&1
\end{array}
\right),&
\left(
\begin{array}{rrr}
-1&0&2\\
0&1&0\\
-2&0&3
\end{array}
\right),&
\left(
\begin{array}{rrr}
1&0&0\\
0&1&0\\
2&0&-1
\end{array}
\right).
\end{eqnarray*}
Then we have the following edge relations
\begin{itemize}
 \item $\bba^{-1}\bbb=\eed=\ggc$,
 \item $\bbd\bbf\bbc\bbe=\eea\\
       =\ggb\ggc^{-1}$,
 \item $\bbd\bbf=\eee=\ggf$.
\end{itemize}

Next, considering
$\Phi_{S_2S_1T_{32}}:\g(3)_{(v_1,v_2)}\to\g(3)_{(v_3,v_4)}$, it follows
that
$\g(3)_{(v_3,v_4)}$,
$\g(3)_{(v_3,v_7)}$ and
$\g(3)_{(v_4,v_7)}$
are generated by
\begin{eqnarray*}
\left(
\begin{array}{rrr}
1&0&0\\
0&1&0\\
-2&2&1
\end{array}
\right),&
\left(
\begin{array}{rrr}
-1&2&0\\
-2&3&0\\
0&0&1
\end{array}
\right),&
\left(
\begin{array}{rrr}
1&0&0\\
2&-1&0\\
0&0&1
\end{array}
\right).
\end{eqnarray*}
Then we have the following edge relations
\begin{itemize}
 \item $\cca^{-1}\ccb=\ddd=\ggd$,
 \item $\ccd\ccf\ccc\cce=\dda\\
       =\gga\ggd^{-1}$,
 \item $\ccd\ccf=\dde=\gge$.
\end{itemize}

Next, considering
$\Phi_{T_{31}T_{32}}:\g(3)_{(v_1,v_2)}\to\g(3)_{(v_5,v_6)}$, it follows
that $\g(3)_{(v_5,v_6)}$ is generated by
\begin{eqnarray*}
\left(
\begin{array}{rrr}
-1&-2&2\\
0&1&0\\
-2&-2&3
\end{array}
\right),&
\left(
\begin{array}{rrr}
1&0&0\\
-2&-1&2\\
-2&-2&3
\end{array}
\right),&
\left(
\begin{array}{rrr}
1&0&0\\
0&1&0\\
2&2&-1
\end{array}
\right).
\end{eqnarray*}
Then we have the following edge relations
\begin{itemize}
 \item $\eeb^{-1}\eea=\ffc\fff\ffe\ffd^{-1}$,
 \item $\eec\eef\eee\eed^{-1}=\ffb^{-1}\ffa$,
 \item $\eec\eee=\ffc\ffe$.
\end{itemize}

Next, considering
$\Phi_{S_2T_{31}T_{32}}:\g(3)_{(v_1,v_2)}\to\g(3)_{(v_4,v_6)}$, it
follows that $\g(3)_{(v_4,v_6)}$ is generated by
\begin{eqnarray*}
\left(
\begin{array}{rrr}
-1&2&-2\\
-2&3&-2\\
0&0&1
\end{array}
\right),&
\left(
\begin{array}{rrr}
1&0&0\\
-2&3&-2\\
-2&2&-1
\end{array}
\right),&
\left(
\begin{array}{rrr}
1&0&0\\
2&-1&2\\
0&0&1
\end{array}
\right).
\end{eqnarray*}
Then we have the following edge relations
\begin{itemize}
 \item $\ddb^{-1}\dda\\
       =\ffb\ffc^{-1}\fff\ffe\ffa\ffd$,
 \item $\ddc\ddf\dde\ddd^{-1}=\ffb^{-1}\ffa^{-1}$,
 \item $\ddc\dde=\ffb\ffc^{-1}\ffe\ffa$.
\end{itemize}

Finally, considering
$\Phi_{S_1S_2T_{31}T_{32}}:\g(3)_{(v_1,v_2)}\to\g(3)_{(v_4,v_5)}$, it
follows that $\g(3)_{(v_4,v_5)}$ is generated by
\begin{eqnarray*}
\left(
\begin{array}{rrr}
3&-2&-2\\
2&-1&-2\\
0&0&1
\end{array}
\right),&
\left(
\begin{array}{rrr}
3&-2&-2\\
0&1&0\\
2&-2&-1
\end{array}
\right),&
\left(
\begin{array}{rrr}
-1&2&2\\
0&1&0\\
0&0&1
\end{array}
\right).
\end{eqnarray*}
Then we have the following edge relations
\begin{itemize}
 \item $\ddb^{-1}\dda^{-1}\\
       =\eeb\eec^{-1}\eef\eee\eea\eed$,
 \item $\ddb\ddc^{-1}\ddf\dde\dda\ddd\\
       =\eeb^{-1}\eea^{-1}$,
 \item $\ddb\ddc^{-1}\dde\dda\\
       =\eeb\eec^{-1}\eee\eea$.
\end{itemize}

Therefore, using Tietze transformations, by Lemma~\ref{g(3)2}, we obtain
the presentation for Proposition~\ref{g(3)}
(For more details see Appendix~\ref{appendix}).
Thus, Theorem~\ref{thm} is valid when $n=3$.

\section{A simplicial complex on which $\g(n)$ acts}\label{gb}

Let $\gb_n(\Z)$ denote the subcomplex of $\B_n(\Z)$ whose
$(k-1)$-simplex $\{x_1,x_2,\dots,x_k\}$ is the set of $k$-vectors
$x_i\in\Z^n$ such that $x_1,x_2,\dots,x_k$ are mutually different column
vectors of a matrix $A\in\g(n)$.
Note that for a vertex $v$, we have $v\equiv{}e_i\mod2$ for some
$1\leq{}i\leq{}n$, where $e_1,e_2,\dots,e_n$ are canonical normal
vectors in $\Z^n$.
For a $(k-1)$-simplex $\Delta=\{x_1,x_2,\dots,x_k\}$, $A\in\g(n)$ is an
{\it extension} of $\Delta$ if each $x_i$ is a column vector of $A$.

In this section, we prove the following proposition.

\begin{prop}\label{gbnz}
For $n\geq4$, the simplicial complex $\gb_n(\Z)$ is simply connected.
\end{prop}

In a proof of this proposition, we will use the idea of Day-Putman
\cite{dp} for proving that $\B_n(\Z)$ is $(n-2)$-connected.

\subsection{Preparation}\

Let $X$ be a simplicial complex.
Then we define followings.
\begin{itemize}
 \item For a simplex $\Delta\in{}X$, $\st_X(\Delta)$ is the subcomplex
       of $X$ whose simplex $\Delta'\in{}X$ satisfies that $\Delta$,
       $\Delta'\subset\Delta''$ for some simplex $\Delta''\in{}X$.
       We also define $\st_X(\emptyset)=X$.
 \item For a simplex $\Delta\in{}X$, $\lk_X(\Delta)$ is the subcomplex
       of $\st_X(\Delta)$ whose simplex $\Delta'\in\st_X(\Delta)$ does
       not intersect $\Delta$.
       We also define $\lk_X(\emptyset)=X$.
\end{itemize}

Here, we prove followings.

\begin{lem}\label{gbnz0}
For $n\geq2$, $\gb_n(\Z)$ is path connected.
\end{lem}

\begin{proof}
We first consider the case $n=2$.
Let $v_0=v_{01}e_1+v_{02}e_2\in\gb_2(\Z)$ be a vertex.
Then there exists a vertex $v_1=v_{11}e_1+v_{12}e_2\in\gb_2(\Z)$ such
that $\{v_0,v_1\}\in\gb_2(\Z)$.
Note that $v_{01}v_{12}-v_{02}v_{11}=\pm1$.
By Euclidean algorithm, we can suppose that $|v_{01}|>|v_{11}|$.
Similarly, there exist vertices
$v_2=v_{21}e_1+v_{22}e_2,\dots,v_k=v_{k1}e_1+v_{k2}e_2\in\gb_2(\Z)$ such
 that $\{v_i,v_{i+1}\}\in\gb_2(\Z)$, $|v_{i1}|>|v_{i+1\;1}|$ for
 $1\leq{}i\leq{}k-1$ and $v_k=e_1$ or $e_2$, for some positive integer
 $k$.
Hence, $\gb_2(\Z)$ is path connected.

Next, we suppose $n\geq3$.
Let $v,w\in\gb_n(\Z)$ be vertices.
Without loss of generality, we suppose $v\equiv{}e_1$ and
 $w\equiv{}e_2\mod2$.
Then there is an extension $A\in\g(n)$ of $v$.
We write $A^{-1}w=\sum_{i=1}^na_ie_i$.
Let $S_{A^{-1}w}=\sum_{i=3}^{n}|a_i|$.
For $3\leq{}i\leq{}n$, if $|a_2|<|a_i|$, there is an integer $u\in\Z$
 such that $|a_2|>|a_i+2ua_2|$.
Then we have that $S_{E_{i2}^uA^{-1}w}<S_{A^{-1}w}$ and
 $E_{i2}^uA^{-1}v=e_1$.
If $|a_2|>|a_i|\neq0$, there is an integer $u'\in\Z$ such that
$|a_2+2u'a_i|<|a_i|$.
In addition, there is an integer $u''\in\Z$ such that
 $|a_2+2u'a_1|>|a_i+2u''(a_2+2u'a_1)|$.
Then we have that $S_{E_{i2}^{u''}E_{2i}^{u'}A^{-1}w}<S_{A^{-1}w}$ and
 $E_{i2}^{u''}E_{2i}^{u'}A^{-1}v=e_1$.
Repeating this operation, we conclude that there exists $B\in\g(n)$ such
 that $S_{Bw}=0$ and $Bv=e_1$.
Note that $Bw$ can be regarded as a vertex in $\gb_2(\Z)$.
Hence, $Bw$ is joined to $e_1$, that is, $Bw$ is joined to $Bv$.
The action of $B^{-1}$ brings the path joining $Bw$ with $Bv$ to the
 path joining $w$ with $v$.
Thus, $\gb_n(\Z)$ is path connected.
\end{proof}

\begin{lem}\label{stlk}
Let $\Delta\in\gb_n(\Z)$ be a $(k-1)$-simplex.
Then we have followings.
\begin{itemize}
 \item $\st_{\gb_n(\Z)}(\Delta)$ is isomorphic to
       $\st_{\gb_n(\Z)}(\{e_1,e_2,\dots,e_k\})$ as a simplicial
       complex.
 \item $\lk_{\gb_n(\Z)}(\Delta)$ is isomorphic to 
       $\lk_{\gb_n(\Z)}(\{e_1,e_2,\dots,e_k\})$ as a simplicial
       complex.
\end{itemize}
\end{lem}

\begin{proof}
For $\Delta=\{x_1,x_2,\dots,x_k\}$, suppose $x_j\equiv{}e_{i(j)}\mod2$.
Let $A\in\g(n)$ be an extension of $\Delta$.
Then restrictions of the action of $A^{-1}$ on $\gb_n(\Z)$
\begin{align*}
A^{-1}|_{\st_{\gb_n(\Z)}(\Delta)}&:\st_{\gb_n(\Z)}(\Delta)\to\st_{\gb_n(\Z)}(\{e_{i(1)},e_{i(2)},\dots,e_{i(k)}\}),\\
A^{-1}|_{\lk_{\gb_n(\Z)}(\Delta)}&:\lk_{\gb_n(\Z)}(\Delta)\to\lk_{\gb_n(\Z)}(\{e_{i(1)},e_{i(2)},\dots,e_{i(k)}\})
\end{align*}
are isomorphisms as a simplicial map.
It is clear that
$\st_{\gb_n(\Z)}(\{e_{i(1)},e_{i(2)},\dots,e_{i(k)}\})$ and
$\lk_{\gb_n(\Z)}(\{e_{i(1)},e_{i(2)},\dots,e_{i(k)}\})$ are respectively
 isomorphic to $\st_{\gb_n(\Z)}(\{e_1,e_2,\dots,e_k\})$ and
 $\lk_{\gb_n(\Z)}(\{e_1,e_2,\dots,e_k\})$.
Thus, we obtain the claim.
\end{proof}

\begin{cor}\label{lk-path}
Let $\Delta\in\gb_n(\Z)$ be a $(k-1)$-simplex.
If $n-k\geq2$, then $\lk_{\gb_n(\Z)}(\Delta)$ is path connected.
\end{cor}

\begin{proof}
By an argument similar to the proof of Lemma~\ref{gbnz0}, we have that
$\lk_{\gb_n(\Z)}(\{e_1,e_2,\dots,e_k\})$ is path connected.
By Lemma~\ref{stlk}, $\lk_{\gb_n(\Z)}(\Delta)$ is also path connected.
\end{proof}

\subsection{Proof of Proposition \ref{gbnz}}\

We suppose $n\geq4$.
Let $\alpha=\{x_i,\{x_i,x_{i+1}\}\mid1\leq{}i\leq{}k,x_{k+1}=x_1\}$ be a
loop on $\gb_n(\Z)$.
We show that $\alpha$ is null-homotopic.

For $v=\sum_{=1}^{n}v_ie_i\in\Z^n$, we define $\R(v)=|v_n|$.
Let $R_{\alpha}=\max\R(x_i)$.

We first prove the next lemma.

\begin{lem}\label{R=0}
For a $1$-simplex $\{v,w\}\in\gb_n(\Z)$ with $\R(v)=\R(w)=0$, we have
 $\{v,w\}\in\lk_{\gb_n(\Z)}(e_n)$.
\end{lem}

\begin{proof}
Note that $v\not\equiv{}w\mod2$.
Suppose that $v\equiv{}e_i$, $w\equiv{}e_j\mod2$ and $i<j$.
Since $\R(v)=\R(w)=0$, we have that $v,w\not\equiv{}e_n\mod2$.
There exists an extension $A=(a_1a_2\cdots{}a_n)\in\g(n)$ of $\{v,w\}$.
Let $S_A=\sum_{l=1}^{n}\R(a_l)$.
Note that $S_A$ is odd.

First, we consider the case $S_A=1$.
Note that $\R(a_l)=0$ for $1\leq{}l\leq{}n-1$ and $\R(a_n)=1$.
Put $a_n=\sum_{i=1}^{n-1}2b_ie_i+\varepsilon{e_n}$, where
 $\varepsilon=\pm1$.
Let
$B=E_{1n}^{b_1}E_{2n}^{b_2}\cdots{E_{n-1n}^{b_{n-1}}}F_n^{\frac{\varepsilon-1}{2}}$.
Then we have $BA=(a_1\cdots{}a_{n-1}e_n)$.
Hence, we have that $\{v,w\}=\{a_i,a_j\}\in\lk_{\gb_n(\Z)}(e_n)$.

Next, we suppose $S_A\geq3$.
Note that there exists $1\leq{}l\leq{}n-1$ with $l\neq{i,j}$ such that
$\R(a_l)\neq0$.
If $\R(a_l)>\R(a_n)$, there exists an integer $u\in\Z$ such that
$\R(a_l+2ua_n)<\R(a_n)$.
Then we have that $AE_{nl}^u$ is an extension of $\{v,w\}$ and that
$S_{AE_{nl}^u}<S_A$.
Similarly, if $\R(a_l)<\R(a_n)$, there exists an integer $u'\in\Z$ such
 that $\R(a_l)>\R(a_n+2u'a_l)$.
Then we have that $AE_{ln}^{u'}$ is an extension of $\{v,w\}$ and that
$S_{AE_{ln}^{u'}}<S_A$.
Repeating this operation, we conclude that there exists an extension
 $A'\in\g(n)$ of $\{v,w\}$ such that $S_{A'}=1$.
Therefore, we have $\{v,w\}\in\lk_{\gb_n(\Z)}(e_n)$.
Thus, we obtain the claim.
\end{proof}

When $R_{\alpha}=0$, by this lemma, we have
$\{x_i,x_{i+1}\}\in\lk_{\gb_n(\Z)}(e_n)$.
Namely, the loop $\alpha$ is in $\lk_{\gb_n(\Z)}(e_n)$.
Since $\lk_{\gb_n(\Z)}(e_n)$ is the subcomplex of
$\st_{\gb_n(\Z)}(e_n)$ and $\st_{\gb_n(\Z)}(e_n)$ is contractible,
$\alpha$ is null-homotopic.
Therefore, we next assume $R_{\alpha}>0$.

Suppose that $R_{\alpha}$ is odd.
There exists $1\leq{}i\leq{}k$ such that $\R(x_i)=R_{\alpha}$.
Since $R_{\alpha}$ is odd, we have that
$x_i\equiv{}e_n$,
$x_{i\pm1}\not\equiv{}e_n\mod2$ and
$\R(x_{i\pm1})<R_{\alpha}$.
By Corollary~\ref{lk-path}, we have that $\lk_{\gb_n(\Z)}(x_i)$ is path
connected.
Since $x_{i\pm1}\in\lk_{\gb_n(\Z)}(x_i)$, there exists a path
$\{y_j,y_l,\{y_j,y_{j+1}\}\mid1\leq{}j\leq{}l-1\}$ on
$\lk_{\gb_n(\Z)}(x_i)$ between $x_{i-1}$ and $x_{i+1}$ such that
$y_1=x_{i-1}$ and $y_l=x_{i+1}$ (see Figure~\ref{odd}).
Since $R_{\alpha}$ is odd and $\R(y_j)$ is even for each $y_j$, there
exists an integer $s_j\in\Z$ such that $\R(y_j')<R_{\alpha}$, where
$y_j'=y_j+2s_jx_i$.
We choose $s_j=0$ if $\R(y_j)<R_{\alpha}$.
When $y_j\equiv{e_t}$, $y_{j+1}\equiv{e_u}\mod2$, for an extension
$A\in\g(n)$ of $\{x_i,y_j,y_{j+1}\}$, we have that
$\{x_i,y_j',y_{j+1}'\}=\{AE_{nt}^{s_j}E_{nu}^{s_{j+1}}e_n,AE_{nt}^{s_j}E_{nu}^{s_{j+1}}e_t,AE_{nt}^{s_j}E_{nu}^{s_{j+1}}e_u\}$.
Hence $\{x_i,y_j',y_{j+1}'\}$ is a $2$-simplex which has an extension
$AE_{nt}^{s_j}E_{nu}^{s_{j+1}}$.
Therefore we have that the path
$\{y_j',y_l',\{y_j',y_{j+1}'\}\mid1\leq{}j\leq{}l-1\}$
between $x_{i-1}$ and $x_{i+1}$ is in $\lk_{\gb_n(\Z)}(x_i)$ (see
Figure~\ref{odd}).
Let
$\alpha'=\alpha\cup\{y_j',y_l',\{y_j',y_{j+1}'\}\mid1\leq{}j\leq{}l-1\}\setminus\{x_i,\{x_i,x_{i\pm1}\}\}$.
Then $\alpha'$ is homotopic to $\alpha$ (see Figure~\ref{odd}).
For all $x_i$ with $\R(x_i)=R_{\alpha}$, applying the same operation, we
conclude that $R_{\beta}<R_{\alpha}$, where $\beta$ is a resulting loop
which is homotopic to $\alpha$.

\begin{figure}[h]
\includegraphics[scale=0.5]{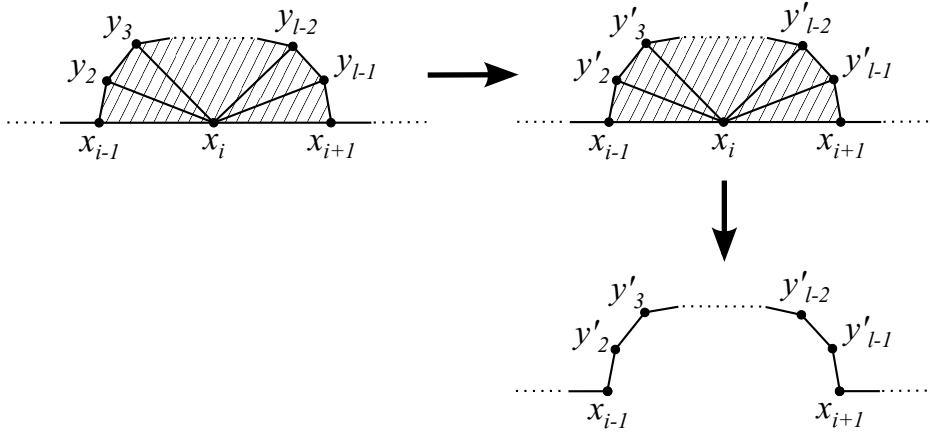}
\caption{The case $R_{\alpha}$ is odd.}\label{odd}
\end{figure}

Next, suppose that $R_{\alpha}$ is even.
There exists $1\leq{}i\leq{}k$ such that $\R(x_i)=R_{\alpha}$.
Since $R_{\alpha}$ is even, we have $x_i\not\equiv{}e_n\mod2$.

\begin{rem}\label{suppose}
Under the assumption $n\geq4$, we may suppose that $\alpha$ satisfies
 all of the following conditions.
\begin{itemize}
 \item $\R(x_{i\pm1})<R_{\alpha}$,
 \item $x_{i\pm1}\not\equiv{}e_n\mod2$,
 \item $x_{i-1}\not\equiv{}x_{i+1}\mod2$.
\end{itemize}
\end{rem}

\begin{proof}
Without loss of generality, we suppose that $x_i\equiv{e_1}\mod2$.
\begin{itemize}
 \item Suppose that $\R(x_{i-1})=R_{\alpha}$.
       Since $R_{\alpha}$ is even we have
       $x_{i-1}\not\equiv{e_n}\mod2$.
       Without loss of generality, we suppose that
       $x_{i-1}\equiv{e_2}\mod2$.
       There exists an extension $A\in\g(n)$ of $\{x_i,x_{i-1}\}$ such
       that $\R(Ae_n)<R_{\alpha}$.
       In fact, if $\R(Ae_n)>R_{\alpha}$, there is an integer $u\in\Z$
       such that $\R(AE_{1n}^ue_n)<R_{\alpha}$.
       Then we choose $AE_{1n}^u$ in place of $A$ as an extension of
       $\{x_i,x_{i-1}\}$.
       (Note that $\R(Ae_n)$ and $\R(AE_{1n}^ue_n)$ are not equal to
       $R_{\alpha}$, since these are odd.)
       Let $y=Ae_n$, and let
       $\alpha'=\alpha\cup\{y,\{x_{i-1},y\},\{y,x_i\}\}\setminus\{\{x_{i-1},x_i\}\}$.
       Then $\alpha'$ is homotopic to $\alpha$.
       Hence, considering $\alpha'$ in place of $\alpha$, we may suppose
       $\R(x_{i-1})<R_{\alpha}$.
       Similarly, we may suppose $\R(x_{i+1})<R_{\alpha}$.
 \item Suppose that $x_{i-1}\equiv{}e_n\mod2$.
       Since $\R(x_{i-1})$ is odd we have $\R(x_{i-1})<R_{\alpha}$.
       There exists an extension $A\in\g(n)$ of $\{x_i,x_{i-1}\}$ such
       that $\R(Ae_2)<\R(x_{i-1})(<R_{\alpha})$.
       In fact, if $\R(Ae_2)>\R(x_{i-1})$, there is an integer $u\in\Z$
       such that $\R(AE_{n2}^ue_2)<\R(x_{i-1})$.
       Then we choose $AE_{n2}^u$ in place of $A$ as an extension of
       $\{x_i,x_{i-1}\}$.
       (Note that $\R(Ae_2)$ and $\R(AE_{n2}^ue_2)$ are not equal to
       $\R(x_{i-1})$, since these are even.)
       Let $y=Ae_2$, and let
       $\alpha'=\alpha\cup\{y,\{x_{i-1},y\},\{y,x_i\}\}\setminus\{\{x_{i-1},x_i\}\}$.
       Then $\alpha'$ is homotopic to $\alpha$.
       Hence, considering $\alpha'$ in place of $\alpha$, we may suppose
       $\R(x_{i-1})<R_{\alpha}$ and $x_{i-1}\not\equiv{}e_n\mod2$.
       Similarly, we may suppose $\R(x_{i+1})<R_{\alpha}$ and
       $x_{i+1}\not\equiv{}e_n\mod2$.
 \item Suppose that $\R(x_{i\pm1})<R_{\alpha}$,
       $x_{i\pm1}\not\equiv{}e_n\mod2$ and
       $x_{i-1}\equiv{}x_{i+1}\mod2$.
       Without loss of generality, we suppose that
       $x_{i\pm1}\equiv{e_2}\mod2$.
       There exists an extension $A\in\g(n)$ of $\{x_i,x_{i-1}\}$ such
       that $\R(Ae_3)\leq\R(x_{i-1})(<R_{\alpha})$.
       In fact, if $\R(Ae_3)>\R(x_{i-1})$, there is an integer $u\in\Z$
       such that $\R(AE_{23}^ue_3)\leq\R(x_{i-1})$.
       Then we choose $AE_{23}^u$ in place of $A$ as an extension of
       $\{x_i,x_{i-1}\}$.
       (Since $Ae_3\not\equiv{x_i,x_{i\pm1},e_n}\mod2$, we need the
       assumption $n\geq4$.)
       Let $y=Ae_3$, and let
       $\alpha'=\alpha\cup\{y,\{x_{i-1},y\},\{y,x_i\}\}\setminus\{\{x_{i-1},x_i\}\}$.
       Then $\alpha'$ is homotopic to $\alpha$.
       Hence, considering $\alpha'$ in place of $\alpha$, we may suppose
       that 
       $\R(x_{i\pm1})<R_{\alpha}$,
       $x_{i\pm1}\not\equiv{}e_n\mod2$ and
       $x_{i-1}\not\equiv{}x_{i+1}\mod2$.
\end{itemize}
\end{proof}

We now suppose that $\alpha$ satisfies the conditions of the above
remark.
Suppose that $x_i\equiv{}e_s$, $x_{i-1}\equiv{}e_t$ and
$x_{i+1}\equiv{}e_u\mod2$, where $s$, $t$ and $u$ are mutually different
 and not equal to $n$.
Since $\{x_{i-1},x_i\}$ is a $1$-simplex in $\gb_n(\Z)$, there is an
extension $B\in\g(n)$ of $\{x_{i-1},x_i\}$.
We write $B^{-1}x_{i+1}=\sum_{j=1}^{n}a_je_j$.
It follows that there exist an even integer $b_u$ and an odd integer
$b_n$ such that $a_ub_n-a_nb_u=\gcd(a_u,a_n)$.
Then we have that
$$
\left(
\begin{array}{cc}
a_u/{\gcd(a_u,a_n)}&b_u\\
a_n/{\gcd(a_u,a_n)}&b_n
\end{array}
\right)^{-1}
\left(
\begin{array}{c}
a_u\\
a_n
\end{array}
\right)=
\left(
\begin{array}{c}
\gcd(a_u,a_n)\\
0
\end{array}
\right).
$$
Let $C\in\g(n)$ be the matrix whose
$(u,u)$ entry is $a_u/{\gcd(a_u,a_n)}$,
$(n,u)$ entry is $a_n/{\gcd(a_u,a_n)}$,
$(u,n)$ entry is $b_u$,
$(n,n)$ entry is $b_n$,
other diagonal entries are $1$ and
other entries are $0$.
Then if we set $A=C^{-1}B^{-1}$, it follows that
$Ax_i=e_s$,
$Ax_{i-1}=e_t$ and
$\R(Ax_{i+1})=0$.

Since $\{e_s,Ax_{i+1}\}$ is a $1$-simplex and $\R(e_s)=\R(Ax_{i+1})=0$,
by Lemma~\ref{R=0}, we have that
$\{e_s,Ax_{i+1}\}\in\lk_{\gb_n(\Z)}(e_n)$.
Therefore, we have that $e_n\in\lk_{\gb_n(\Z)}(\{e_s,Ax_{i+1}\})$.
In addition, it is clear that $e_n\in\lk_{\gb_n(\Z)}(\{e_s,e_t\})$.
Hence, we have that $A^{-1}e_n\in\lk_{\gb_n(\Z)}(\{x_i,x_{i\pm1}\})$
(see Figure~\ref{even}).
Then, there exists an integer $l$ such that $\R(x_i')<R_{\alpha}$, where
$x_i'=A^{-1}e_n+2lx_i$.
We have also that $x_i'\in\lk_{\gb_n(\Z)}(\{x_i,x_{i\pm1}\})$
(see Figure~\ref{even}).
Let
$\alpha'=\alpha\cup\{\{x_i'\},\{x_i',x_{i\pm1}\}\}\setminus\{x_i,\{x_i,x_{i\pm1}\}\}$.
Then $\alpha'$ is homotopic to $\alpha$ (see Figure~\ref{even}).
Similar to the case $R_{\alpha}$ is odd, for all $x_i$ with
$\R(x_i)=R_{\alpha}$, applying the same operation, we conclude that
$R_{\beta}<R_{\alpha}$, where $\beta$ is a resulting loop which is
homotopic to $\alpha$.

\begin{figure}[h]
\includegraphics[scale=0.5]{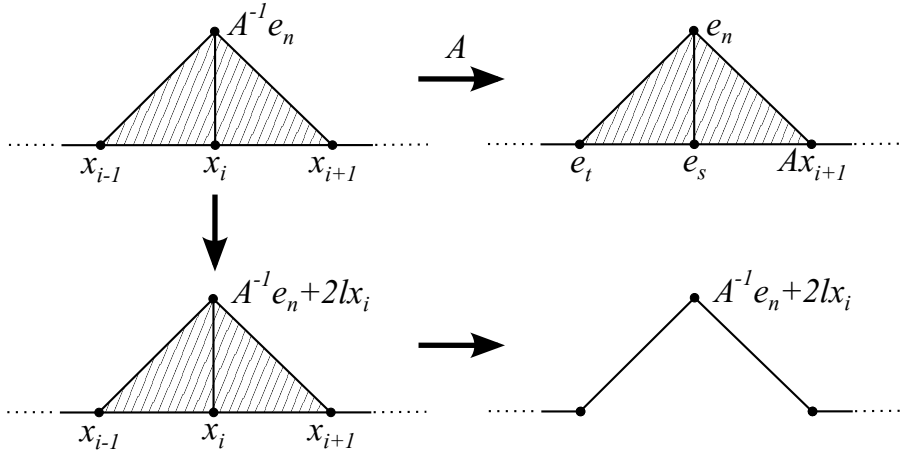}
\caption{The case $R_{\alpha}$ is even.}\label{even}
\end{figure}

Repeating this operation until $R_{\alpha}=0$, we conclude that the loop
$\alpha$ on $\gb_n(\Z)$ is null homotopic.
Thus, $\gb_n(\Z)$ is simply connected.

\section{Proof of Theorem \ref{thm}}\label{proof}

We first prove the next proposition.

\begin{lem}\label{gnZ}
For any $n\geq4$, $\g(n)$ is isomorphic to the quotient of
$\underset{1\leq{i}\leq{n}}{\ast}\g(n)_{e_i}$ by the normal subgroup
 generated by edge relators.
\end{lem}

\begin{proof}
For a $(k-1)$-simplex $\Delta=\{x_1,x_2,\dots,x_k\}\in\gb_n(\Z)$ with
 $x_j\equiv{}e_{i(j)}\mod2$, let $A\in\g(n)$ be an extension of
 $\Delta$.
Then we have $A^{-1}\cdot\Delta=\{e_{i(1)},e_{i(2)},\dots,e_{i(k)}\}$.
Therefore, we have
$$
\g(n)\backslash\gb_n(\Z)=\{\{e_{i(1)},e_{i(2)},\dots,e_{i(k)}\}\mid1\leq{}k\leq{}n,1\leq{}i(1)<i(2)<\cdots<i(k)\leq{}n\}.
$$
It is clear that $\g(n)\backslash\gb_n(\Z)$ is contractible.
Note that the action of $\g(n)$ on $\gb_n(\Z)$ is without rotation.

We first set followings.
\begin{itemize}
 \item $T=\{(e_1,e_i)\mid2\leq{}i\leq{}n\}$.
 \item $E=\{(e_i,e_j)\mid1\leq{}i<j\leq{}n\}$.
 \item $F=\{(e_i,e_j,e_k)\mid1\leq{}i<j<k\leq{}n\}$.
 \item For $e\in{}E$, we choose $g_e=1$, and write $g_e=g_{ij}$ when
       $e=(e_i,e_j)$.
 \item For $\tau=(e_i,e_j,e_k)\in{}F$, let
       $g_{\tau}=g_{ij}g_{jk}g_{ik}^{-1}$.
\end{itemize}
Then, since $\gb_n(\Z)$ is simply connected, it follows from
 Theorem~\ref{Brown} that $\g(n)$ is isomorphic to the quotient of
$\left(\underset{1\leq{}i\leq{}n}{\ast}\g(n)_{e_i}\right)\ast\left(\underset{1\leq{}i<j\leq{}n}{\ast}\la\hat{g}_{ij}\ra\right)$
 by the normal subgroup generated by followings
\begin{enumerate}
 \item $\hat{g}_{1i}$, where $2\leq{}i\leq{}n$,
 \item $\hat{g}_{ij}^{-1}X_{e_i}\hat{g}_{ij}X_{e_j}^{-1}$, where
       $1\leq{}i<j\leq{}n$ and $X\in\g(n)_{(e_i,e_j)}$,
 \item $\hat{g}_{\tau}g_{\tau}^{-1}$, where $\tau\in{}F$.
\end{enumerate}
Since $g_{\tau}=1$, the relation $\hat{g}_{\tau}g_{\tau}^{-1}$ is
 equivalent to the relation $\hat{g}_{ij}\hat{g}_{jk}=\hat{g}_{ik}$ if
 $\tau=(e_i,e_j,e_k)$.
By relations $\hat{g}_{1i}=1$, we have the relation $\hat{g}_{ij}=1$ for
$1\leq{}i<j\leq{}n$.
Thus, we obtain the claim.
\end{proof}

Note that for $e=(e_s,e_t)$, $\g(n)_{e}$ is generated by $(E_{ij})_e$
and $(F_j)_e$ for $1\leq{}i,j\leq{}n$ with $j\neq{}s,t$.
Hence, we have edge relations
\begin{itemize}
 \item $(E_{ij})_{e_s}=(E_{ij})_{e_t}$,
 \item $(F_j)_{e_s}=(F_j)_{e_t}$.
\end{itemize}

Since we already obtained presentations of $\g(2)$ and $\g(3)$, from
Lemma~\ref{gnZ} and Remark~\ref{sesrem}, we obtain the presentation of
$\g(n)$ for $n\geq4$, by induction on $n$.

Thus, we complete the proof of Theorem~\ref{thm}.

\appendix
\section{}\label{appendix}

In this section, we check Tietze transformations of
Subsection~\ref{edge}.

Let $\hg$ denote the quotient of
$\underset{1\leq{}i\leq7}{\ast}\g(3)_{v_i}$ by the normal subgroup
generated by edge relators.
By the edge relations of Subsection~\ref{edge}, we have the following
relations, in $\hg$,
\begin{enumerate}
 \item \begin{itemize}
	\item $\bbb=\aac$,
	\item $\bbc=\aab$,
	\item $\bbf=\aaf$,
       \end{itemize}
 \item \begin{itemize}
	\item $\cca=\bbd$,
	\item $\ccb=\aad$,
	\item $\ccc=\aaa$,
	\item $\ccd=\bba$,
	\item $\cce=\bbe$,
	\item $\ccf=\aae$,
       \end{itemize}
 \item \begin{itemize}
	\item $\dda=\bba\aae\aaa\bbe$,
	\item $\ddb=\aab\aac$,
	\item $\ddc=\aac$,
	\item $\ddd=\bbd^{-1}\aad$,
	\item $\dde=\bba\aae$,
	\item $\ddf=\aaf$,
       \end{itemize}
 \item \begin{itemize}
	\item $\eea=\bbd\aaf\aab\bbe$,
	\item $\eeb=\aaa\aad$,
	\item $\eec=\aad$,
	\item $\eed=\bba^{-1}\aac$,
	\item $\eee=\bbd\aaf$,
	\item $\eef=\aae$,
       \end{itemize}
 \item \begin{itemize}
	\item $\ffa=\aad\aaf\aac\aae$,
	\item $\ffb=\bba\bbd$,
	\item $\ffc=\bbd$,
	\item $\ffd=\aaa^{-1}\aab$,
	\item $\ffe=\aad\aaf$,
	\item $\fff=\bbe$,
       \end{itemize}
 \item \begin{itemize}
	\item $\gga=\bba\aae\aaa\bbe\bbd^{-1}\aad$,
	\item $\ggb=\bbd\aaf\aab\bbe\bba^{-1}\aac$,
	\item $\ggc=\bba^{-1}\aac$,
	\item $\ggd=\bbd^{-1}\aad$,
	\item $\gge=\bba\aae$,
	\item $\ggf=\bbd\aaf$.
       \end{itemize}
\end{enumerate}
Using Tietze transformations, we obtain a presentation of $\hg$ whose
generators are
$\aaa$,
$\aab$,
$\aac$,
$\aad$,
$\aae$,
$\aaf$,
$\bba$,
$\bbd$ and
$\bbe$.
To avoid complication of notations, we rewrite $X=X_{v_i}$.
Then we have a finite presentation of $\hg$ with generators
$\xaaa$,
$\xaab$,
$\xaac$,
$\xaad$,
$\xaae$,
$\xaaf$,
$\xbba$,
$\xbbd$ and
$\xbbe$,
and with the following relators
\begin{enumerate}
 \item[(1.1)] $\xaae^2$,
	      $\xaaf^2$,
 \item[(1.2)] $(\xaaa\xaae)^2$,
	      $(\xaab\xaaf)^2$,
	      $(\xaac\xaae)^2$,
	      $(\xaac\xaaf)^2$,
	      $(\xaad\xaae)^2$,
	      $(\xaad\xaaf)^2$,
	      $(\xaae\xaaf)^2$,
 \item[(1.3)] $[\xaaa,\xaab]$,
	      $[\xaaa,\xaad]$,
	      $[\xaaa,\xaaf]$,
	      $[\xaab,\xaac]$,
	      $[\xaab,\xaae]$,
	      $[\xaac,\xaaa]\xaab^2$,
	      $[\xaad,\xaab]\xaaa^2$,
 \item[(2.1)] $\xbbe^2$,
 \item[(2.2)] $(\xaab\xbbe)^2$,
	      $(\xbba\xbbe)^2$,
	      $(\xbbd\xbbe)^2$,
	      $(\xbbd\xaaf)^2$,
	      $(\xbbe\xaaf)^2$,
 \item[(2.3)] $[\xbba,\xaac]$,
	      $[\xbba,\xbbd]$,
	      $[\xbba,\xaaf]$,
	      $[\xaac,\xbbe]$,
	      $[\xaab,\xbba]\xaac^2$,
	      $[\xbbd,\xaac]\xbba^2$,
 \item[(3.2)] $(\xaaa\xbbe)^2$,
	      $(\xbba\xaae)^2$,
	      $(\xbbe\xaae)^2$,
 \item[(3.3)] $[\xbbd,\xaad]$,
	      $[\xbbd,\xaae]$,
	      $[\xaad,\xbbe]$,
	      $[\xaaa,\xbbd]\xaad^2$,
	      $[\xbba,\xaad]\xbbd^2$,		
 \item[(4.3)] $[\xbbd^{-1}\xaad,\xaab\xaac](\xbba\xaae\xaaa\xbbe)^2$,
 \item[(5.3)] $[\xbba^{-1}\xaac,\xaaa\xaad](\xbbd\xaaf\xaab\xbbe)^2$,
 \item[(6.3)] $[\xaaa^{-1}\xaab,\xbba\xbbd](\xaad\xaaf\xaac\xaae)^2$,
 \item[(7.3)] \begin{enumerate}
	       \item $[\xbba\xaae\xaaa\xbbe\xbbd^{-1}\xaad,\xbbd\xaaf\xaab\xbbe\xbba^{-1}\xaac]$,
	       \item $[\xbba^{-1}\xaac,\xbba\xaae\xaaa\xbbe\xbbd^{-1}\xaad](\xbbd\xaaf\xaab\xbbe\xbba^{-1}\xaac)^2$,
	       \item $[\xbbd^{-1}\xaad,\xbbd\xaaf\xaab\xbbe\xbba^{-1}\xaac](\xbba\xaae\xaaa\xbbe\xbbd^{-1}\xaad)^2$.
	      \end{enumerate}
\end{enumerate}

Let $X$, $Y$ and $Z$ be
\begin{eqnarray*}
X&=&\{(F_iF_j)^2,(E_{ij}F_i)^2,(E_{ij}F_j)^2,[E_{ij},F_k]\mid\{i,j,k\}=\{1,2,3\}\},\\
Y&=&\{[E_{ij},E_{ik}],[E_{ij},E_{kj}]\mid\{i,j,k\}=\{1,2,3\}\},\\
Z&=&\{[E_{ij},E_{ki}]E_{kj}^2\mid\{i,j,k\}=\{1,2,3\}\}.
\end{eqnarray*}
We show that relators (4.3), (5.3), (6.3) and (b), (c) of (7.3) are
obtained from relators $X$, $Y$, $Z$ and (a) of (7.3).
In transformation, the notation ``$\equiv$'' means conjugation.
An underline means applying relators $Y$, $Z$ or (a) of (7.3).

\begin{lem}\label{7.3}
Under relators {\rm (1.-), (2.-), (3.-)} and conjugation,
\begin{enumerate}
 \item the relator {\rm (a)} of {\rm (7.3)} is equivalent to the relator
       $(\eji\eij^{-1}\ekj^{-1}\ejk\eik\eki^{-1})^2$,
 \item relators {\rm (b)} and {\rm (c)} of {\rm (7.3)} are equivalent to
       the relator\\
       $\ekj^{-1}\eij\eji^{-1}\ejk^{-1}\ekj\eij^{-1}\eji\ejk\eik^{-1}\eki\eik^{-1}\eki$,
\end{enumerate}
where $(j,k)=(2,3)$ or $(3,2)$.
\end{lem}

\begin{proof}
\begin{enumerate}
 \item At first, we delete words $F_1$, $F_2$ and $F_3$, using relators
       $X$, and then transform as follows.
       \begin{eqnarray*}
	& &[\eji\ffj\eij\ffi\eki^{-1}\ekj,\eki\ffk\eik\ffi\eji^{-1}\ejk]\\
        &=&(\eji\ffj\eij\ffi\underset{Y}{\underline{\eki^{-1}\ekj)
	   (\eki}}\ffk\eik\ffi\eji^{-1}\ejk)\\
	& &\cdot(\ekj^{-1}\eki\ffi\eij^{-1}\ffj\underset{Y}{\underline{\eji^{-1})
	   (\ejk^{-1}\eji}}\ffi\eik^{-1}\ffk\eki^{-1})\\
	&\underset{X}{=}&\eji\eij^{-1}\ekj^{-1}\cdot
	   \underset{Y}{\underline{\eik\eji\ejk}}\cdot
	   \ekj^{-1}\eki^{-1}\eij^{-1}\cdot
	   \ejk\eik\eki^{-1}\\
	&=&\eji\eij^{-1}\ekj^{-1}\cdot
	   \ejk\eik\underset{Y}{\underline{\eji\cdot
	   \ekj^{-1}\eki^{-1}}}\eij^{-1}\cdot
	   \ejk\eik\eki^{-1}\\
	&=&\eji\eij^{-1}\ekj^{-1}\cdot
	   \ejk\eik\eki^{-1}\eji\cdot
	   \underset{Y}{\underline{\ekj^{-1}\eij^{-1}}}\cdot
	   \ejk\eik\eki^{-1}\\
	&=&\eji\eij^{-1}\ekj^{-1}
	   \ejk\eik\eki^{-1}\cdot\eji
	   \eij^{-1}\ekj^{-1}
	   \ejk\eik\eki^{-1}\\
	&=&(\eji\eij^{-1}\ekj^{-1}\ejk\eik\eki^{-1})^2.
       \end{eqnarray*}
       Thus, we obtain the claim.
 \item Similarly, we delete words $F_1$, $F_2$ and $F_3$ as follows.
       \begin{eqnarray*}
        & &[\eji^{-1}\ejk,\eji\ffj\eij\ffi\eki^{-1}\ekj]\\
	&=&\ejk^{-1}\eji\cdot\underset{Y}{\underline{\ekj^{-1}\eki}}\ffi\eij^{-1}\ffj\eji^{-1}\cdot
	   \underset{Y}{\underline{\eji^{-1}\ejk\cdot\eji}}\ffj\eij\ffi\eki^{-1}\ekj\\
	&\underset{X}{=}&\ejk^{-1}\eji\cdot\eki\ekj^{-1}\eij\eji^{-1}\cdot\ejk^{-1}\cdot\eij^{-1}\eki^{-1}\ekj,\\
        & &(\eki\ffk\eik\ffi\eji^{-1}\ejk)^2\\
	&=&\eki\ffk\eik\ffi\eji^{-1}\ejk\cdot\eki\ffk\eik\ffi\eji^{-1}\ejk\\
	&\underset{X}{=}&\eki\eik^{-1}\eji\ejk^{-1}\cdot\eki\eik^{-1}\eji^{-1}\ejk.
       \end{eqnarray*}
       We next calculate
       \begin{eqnarray*}
	& &[\eji^{-1}\ejk,\eji\ffj\eij\ffi\eki^{-1}\ekj](\eki\ffk\eik\ffi\eji^{-1}\ejk)^2\\
	&=&\ejk^{-1}\eji\eki\ekj^{-1}\eij\eji^{-1}\ejk^{-1}\underset{Y}{\underline{\eij^{-1}\eki^{-1}\ekj
	   \cdot\eki}}\;\underset{Z}{\underline{\eik^{-1}\eji\ejk^{-1}}}\\
	& &\cdot\eki\eik^{-1}\eji^{-1}\ejk\\
	&\equiv&\ekj^{-1}\eij\eji^{-1}\ejk^{-1}\ekj\eij^{-1}\eji\ejk\eik^{-1}\eki\eik^{-1}\eki.
       \end{eqnarray*}
       Thus, we obtain the claim.
\end{enumerate}
\end{proof}

\begin{prop}\label{7.3(b)(c)}
Each of relators {\rm (b)} and {\rm (c)} of {\rm (7.3)} is obtained from
 relators $X$, $Y$, $Z$ and {\rm (a)} of {\rm (7.3)}.
\end{prop}

\begin{proof}
Let $(j,k)=(2,3)$ or $(3,2)$.
We calculate
\begin{eqnarray*}
1&=&\eji\eij^{-1}\ekj^{-1}\underset{Y}{\underline{\ejk\eik}}\eki^{-1}\cdot
   \eji\eij^{-1}\ekj^{-1}\ejk\eik\eki^{-1}\\
&=&\eji\underset{Z}{\underline{\eij^{-1}\ekj^{-1}\eik}}\;\underset{Z}{\underline{\ejk\eki^{-1}\eji}}\eij^{-1}\ekj^{-1}\ejk\eik\eki^{-1}\\
&=&\eji\eik\underset{Z}{\underline{\eij\ekj^{-1}\eki^{-1}}}\eji^{-1}\ejk\eij^{-1}\ekj^{-1}\ejk\eik\eki^{-1}\\
&=&\eji\eik\eki^{-1}\ekj\eij\eji^{-1}\ejk\eij^{-1}\ekj^{-1}\ejk\eik\eki^{-1}\\
&\equiv&\ekj\eij\eji^{-1}\ejk\underset{Y}{\underline{\eij^{-1}\ekj^{-1}}}\ejk\eik\underset{Y}{\underline{\eki^{-1}\eji}}\eik\eki^{-1}\\
&=&\ekj\eij\eji^{-1}\ejk\ekj^{-1}\eij^{-1}\underset{Z}{\underline{\ejk\eik\eji}}\eki^{-1}\eik\eki^{-1}\\
&=&(\ejk\eik\eki^{-1}\underset{\textrm{(a) of (7.3)}}{\underline{\eki\eik^{-1}\ejk^{-1})\ekj\eij\eji^{-1}}}\ejk\ekj^{-1}\eij^{-1}\eji\ejk^{-1}\eik\eki^{-1}\eik\eki^{-1}\\
&=&\ejk\eik\eki^{-1}\eji\eij^{-1}\ekj^{-1}\ejk\eik\eki^{-1}\ejk\ekj^{-1}\eij^{-1}\eji\ejk^{-1}\eik\eki^{-1}\eik\eki^{-1}\\
&\equiv&\ekj^{-1}\cdot\ejk\eik\eki^{-1}\eji\eij^{-1}\ekj^{-1}\ejk\eik\eki^{-1}(\eji\eij^{-1}\eij\eji^{-1})\\
& &\cdot\ejk\ekj^{-1}\eij^{-1}\eji\ejk^{-1}\eik\eki^{-1}\eik\eki^{-1}\cdot\ekj\\
&=&\underset{\textrm{(a) of (7.3)}}{\underline{(\ekj^{-1}\ejk\eik\eki^{-1}\eji\eij^{-1})^2}}\eij\eji^{-1}\ejk\ekj^{-1}\eij^{-1}\eji\ejk^{-1}\eik\eki^{-1}\eik\eki^{-1}\ekj\\
&=&\eij\eji^{-1}\ejk\ekj^{-1}\eij^{-1}\eji\ejk^{-1}\eik\eki^{-1}\eik\eki^{-1}\ekj\\
&\equiv&\ffk\cdot\ekj\eij\eji^{-1}\ejk\ekj^{-1}\eij^{-1}\eji\ejk^{-1}\eik\eki^{-1}\eik\eki^{-1}\cdot\ffk\\
&\underset{X}{=}&\ekj^{-1}\eij\eji^{-1}\ejk^{-1}\ekj\eij^{-1}\eji\ejk\eik^{-1}\eki\eik^{-1}\eki.
\end{eqnarray*}
By Lemma~\ref{7.3}, we obtain the claim.
\end{proof}

\begin{prop}\label{456}
Each of relators {\rm (4.3), (5.3)} and {\rm (6.3)} is obtained from
 other relators and conjugation.
\end{prop}

\begin{proof}
We first consider relators (4.3) and (5.3).
Let $(j,k)=(2,3)$ or $(3,2)$.
\begin{eqnarray*}
& &[\eji^{-1}\ejk,\eij\ekj](\eki\ffk\eik\ffi)^2\\
&=&\ejk^{-1}\eji\cdot\underset{Y}{\underline{\ekj^{-1}\eij^{-1}}}\cdot\underset{Y}{\underline{\eji^{-1}\ejk}}\cdot\eij\ekj\cdot
   \eki\ffk\eik\ffi\cdot\eki\ffk\eik\ffi\\
&\underset{X}{=}&\ejk^{-1}\eji\eij^{-1}\ekj^{-1}\ejk\eji^{-1}\eij\ekj\eki\eik^{-1}\eki\eik^{-1}\\
&\equiv&\ffi(\eki\eik^{-1}\eki\eik^{-1}\ejk^{-1}\eji\eij^{-1}\ekj^{-1}\ejk\eji^{-1}\eij\ekj)\ffi\\
&\underset{X}{=}&\eki^{-1}\eik\eki^{-1}\eik\ejk^{-1}\eji^{-1}\eij\ekj^{-1}\ejk\eji\eij^{-1}\ekj\\
&=&(\ekj^{-1}\eij\eji^{-1}\ejk^{-1}\ekj\eij^{-1}\eji\ejk\eik^{-1}\eki\eik^{-1}\eki)^{-1}.
\end{eqnarray*}

We next consider the relator (6.3).
\begin{eqnarray*}
& &[\xij^{-1}\xik,\xji\xki](\xkj\xxk\xjk\xxj)^2\\
&=&\xik^{-1}\xij\cdot\xki^{-1}\xji^{-1}\cdot\xij^{-1}\xik\cdot\underset{Z}{\underline{\xji\xki\cdot
   \xkj}}\xxk\xjk\xxj\cdot\xkj\xxk\xjk\xxj\\
&\underset{X}{=}&\xik^{-1}\xij\xki^{-1}\xji^{-1}\xij^{-1}\xik\xki^{-1}\xkj\underset{Y}{\underline{\xji\xjk^{-1}}}\xkj\xjk^{-1}\\
&\equiv&\underset{Z}{\underline{\xjk^{-1}\xik^{-1}\xij}}\xki^{-1}\xji^{-1}\xij^{-1}\xik\xki^{-1}\xkj\xjk^{-1}\xji\xkj\\
&=&\xik\xij\underset{Z}{\underline{\xjk^{-1}\xki^{-1}\xji^{-1}}}\xij^{-1}\xik\xki^{-1}\xkj\xjk^{-1}\xji\xkj\\
&=&\xik\xij\xki^{-1}\xji\underset{Z}{\underline{\xjk^{-1}\xij^{-1}\xik}}\xki^{-1}\xkj\xjk^{-1}\xji\xkj\\
&=&\xik\xij\xki^{-1}\xji\xij^{-1}\xjk^{-1}\xik^{-1}\xki^{-1}\xkj\xjk^{-1}\xji\xkj\\
&\equiv&\underset{Z}{\underline{\xij^{-1}\xjk^{-1}\xik^{-1}}}\xki^{-1}\xkj\xjk^{-1}\xji\underset{Z}{\underline{\xkj\xik\xij}}\xki^{-1}\xji\\
&=&\xjk^{-1}\xik\underset{Z}{\underline{\xij^{-1}\xki^{-1}\xkj}}\;\underset{Z}{\underline{\xjk^{-1}\xji\xik}}\;\underset{Z}{\underline{\xkj\xij^{-1}\xki^{-1}}}\xji\\
&=&\xjk^{-1}\xik\xki^{-1}\xkj^{-1}\underset{Z}{\underline{\xij^{-1}\xjk\xik}}\;\underset{Z}{\underline{\xji\xkj^{-1}\xki^{-1}}}\xij^{-1}\xji\\
&=&\xjk^{-1}\xik\xki^{-1}\xkj^{-1}\xjk\xik^{-1}\underset{Z}{\underline{\xij^{-1}\xki\xkj^{-1}}}\xji\xij^{-1}\xji\\
&=&\xjk^{-1}\xik\xki^{-1}\xkj^{-1}\xjk\xik^{-1}\xki\xkj\xij^{-1}\xji\xij^{-1}\xji.
\end{eqnarray*}
By Lemma~\ref{7.3}, each of relators (4.3), (5.3) and (6.3) is obtained
 from relators (1.-), (2.-), (3.-) and (b), (c) of (7.3).
Thus, we obtain the claim.
\end{proof}

\section*{Acknowledgement}
The author would like to express his thanks to Dan Margalit, Andrew
Putman and Neil Fullarton for informing the author about their results
including their finite presentations of $\g(n)$, Susumu Hirose and
Masatoshi Sato for their valuable suggestions and useful comments.



\end{document}